\documentclass[a4paper]{amsart}

\usepackage[utf8]{inputenc}
\usepackage[british]{babel}
\usepackage{amsfonts,amsmath,amssymb,amsthm,mathtools,esint}
\usepackage[dvipsnames]{xcolor}
\usepackage{tikz}
\usepackage{pgfplots}
\pgfplotsset{compat=1.3}
\usepackage{caption}
\usepackage[justification=centering]{caption}
\usepackage{enumitem}

\usepackage{filecontents}
\usepackage[hidelinks]{hyperref}
\usepackage{cleveref}
\crefname{subsection}{subsection}{subsections}
\numberwithin{equation}{section}

\usepackage[backend=bibtex, giveninits=true, isbn=false, doi=false, url=false, maxnames = 4]{biblatex}
\DeclareFieldFormat[article]{title}{{#1}} 
\DeclareFieldFormat[book]{title}{{#1}}
\DeclareFieldFormat[inproceedings]{title}{{#1}}
\DeclareFieldFormat[incollection]{title}{{#1}}
\DeclareFieldFormat{pages}{{#1}} 
\AtBeginBibliography{\footnotesize}
\newcommand{\vertiii}[1]
{\vert\kern-0.25ex\vert\kern-0.25ex\vert #1 
\vert\kern-0.25ex\vert\kern-0.25ex\vert}
	
\newcommand{\R}{\mathbb{R}}
\renewcommand{\d}[1]{\,\mathrm{d}#1}
\renewcommand{\div}{\mathrm{div}}
\newcommand{\I}{\mathrm{I}}

\newcommand{\D}{\mathrm{D}}

\newcommand{\Mcal}{\mathcal{M}}

\newcommand{\s}{\mathrm{s}}
\newcommand{\pw}{\mathrm{pw}}

\newtheorem{theorem}{Theorem}[section]
\newtheorem{lemma}[theorem]{Lemma}
\newtheorem{corollary}[theorem]{Corollary}
\theoremstyle{remark}
\newtheorem{remark}[theorem]{Remark}
\theoremstyle{definition}

\newtheoremstyle{assumptionnoperiod}
{1ex plus 0.5ex minus 0.5ex}   
{1ex plus 0.5ex minus 0.5ex}   
{\normalfont}  
{0pt}       
{\bfseries} 
{}          
{5pt plus 1pt minus 1pt} 
{}          
\theoremstyle{assumptionnoperiod}
\newtheorem{assumption}{}

\newcounter{cnst}

\begin{document}

\title[Lower eigenvalue bounds with hybrid high-order methods]{Lower eigenvalue bounds with hybrid high-order methods}
\author[N.~T.~Tran]
       {Ngoc Tien Tran}
\thanks{The author received funding from
		the European Union's Horizon 2020 research and innovation 
		programme (project RandomMultiScales, grant agreement No. 865751).}
\address[N.~T.~Tran]{%
         Institut f\"ur Mathematik,
         Universit\"at Augsburg,
         86159 Augsburg, Germany}
         \email{ngoc.tien.tran@uni-a.de}
\date{\today}




\begin{abstract}
	This paper proposes hybrid high-order eigensolvers for the computation of guaranteed lower eigenvalue bounds. These bounds display higher order convergence rates and are accessible to adaptive mesh-refining algorithms. The involved constants arise from local embeddings and are available for all polynomial degrees. A wide range of applications is possible including the linear
	elasticity and Steklov eigenvalue problem.
\end{abstract}
\maketitle
\section{Introduction}
\subsection{Motivation}
Guaranteed error control for the approximation of eigenvalues associated with positive definite
partial differential equations requires upper and lower bounds of the exact eigenvalue $\lambda(j)$. While upper bounds can be achieved by conforming methods using the Rayleigh-Ritz principle, lower bounds are more challenging to obtain.
Lower bounds with nonconforming methods under sufficiently small mesh-size conditions can be derived from the asymptotic analysis of, e.g., \cite{YangZhangLin2010,LuoLinXie2012}.
This paper derives guaranteed lower eigenvalue bounds (GLB) without any condition on the mesh.

Using the nonconforming Crouzeix-Raviart finite element method (FEM), \cite{CarstensenGedicke2014} derived GLB for the Laplace eigenvalue problem, with extension to the Morley-FEM for the biharmonic eigenvalue problem in \cite{CarstensenGallistl2014}. Both contributions \cite{CarstensenGedicke2014,CarstensenGallistl2014} required a separation condition for higher eigenvalues, which is, in fact, not necessary as observed in \cite{Liu2015}.
At a similar time to \cite{CarstensenGedicke2014}, \cite{LiuOishi2013} derived GLB for the Laplace eigenvalue problem using conforming finite elements via the hyper-circle principle.
While additional computational effort is required to compute the explicit constants in the GLB, this approach is more general and can be applied to, e.g., the Maxwell and Steklov eigenvalue problems \cite{YouXieLiu2019,GallistlOlkhovskiy2023}.
Recently, mixed methods \cite{Gallistl2023} provide GLB for a large range of examples, where the involved constants are independent of the explicit discretization.
The bounds of all aforementioned methods are postprocessings of the computed eigenvalue $\lambda_h(j)$ of the form
\begin{align*}
	\frac{\lambda_h(j)}{1 + \delta \lambda_h(j)} \leq \lambda(j)
\end{align*}
for the $j$-th exact eigenvalue.
They are restricted to lowest-order convergence because, in practice, the parameter $\delta$ has a fixed scaling in terms of maximal mesh-size arising from some local embeddings. Thus, higher-order methods cannot provide any improvement and uniform mesh-refining algorithms may even outperform adaptive computations in certain examples on nonconvex domains \cite{CarstensenPuttkammer2023}.

A remedy is \cite{CarstensenZhaiZhang2020} with a hybridizable discontinuous Galerkin (HDG) eigensolver plus the Lehrenfeld-Sch\"oberl stabilization, where the computed eigenvalue $\lambda_h(j) \leq \lambda(j)$ is a direct GLB of $\lambda(j)$ (without further postprocessing) under an explicit assumption on the maximal mesh-size.
This makes the bound accessible to higher-order discretizations and adaptivity.
The method of \cite{CarstensenZhaiZhang2020}, the hybridizable methods \cite{CarstensenErnPuttkammer2021,CarstensenGraessleTran2024} mentioned below, and the proposed method of this paper can also be interpreted as an HDG \cite{CockburnGopalakrishnanLazarov2009}, a weak Galerkin (WG) \cite{WangYe2014}, or an hybrid high-order (HHO) method \cite{DiPietroErnLemaire2014,DiPietroErn2015}.
In fact, WG methods can be embedded into the HDG methodology \cite[Section 6.6]{Cockburn2016} and the reconstructions therein have close relationship to the HHO methodology. We also refer to
\cite{CockburnDiPietroErn2016} for explicit links between HDG and HHO methods and the monograph \cite{CicuttinErnPignet2021} for further discussions on the connection between HHO and WG methods.

The HHO methods of \cite{CarstensenErnPuttkammer2021,CarstensenGraessleTran2024} utilize multiple discrete spaces to reconstruct the gradient. As a result, an additional selectable parameter allows for more flexibility and lead, in numerical experiments, to improved GLB \cite{CarstensenGraesslePirch2025} with $p$-robust constants in \cite{CarstensenGraessleTran2024}.
Optimal convergence rates of adaptive mesh-refining algorithms can be proven for a related lowest-order extra-stabilized Crouzeix-Raviart FEM \cite{CarstensenPuttkammer2024}.
However, \cite{CarstensenZhaiZhang2020,CarstensenErnPuttkammer2021,CarstensenGraessleTran2024} are only concerned with the Laplace eigenvalue problem and the involved constants therein are \emph{merely} guaranteed for lowest-order discretizations.

This paper exploits the fact that the Poincar\'e constant is explicitly known in convex sets. The degrees of freedom associated with the kernel of $\nabla$ can be locally eliminated in the definition of the potential reconstruction $\mathcal{R}_h$ in the HHO methodology.
Therefore, the HHO methods of this paper only utilize the single  reconstruction operator $\mathcal{R}_h$. This allows for guaranteed constants independent of the polynomial degree in the GLB below. 
As an extra benefit, hanging nodes are allowed as long as the mesh is partitioned into convex cells for additional flexibility in the mesh design.
Since the presented approach relies on intrinsic features of the HHO methodology, it applies to a wide range of examples including linear elasticity and Steklov eigenvalue problems.
To be precise, we identify
three abstract conditions for the GLB
\begin{align}\label{ineq:lower-bound}
	\mathrm{GLB}(j) \coloneqq \min\{1, 1/(\alpha + \beta \lambda_h(j))\}\lambda_h(j) \leq \lambda(j)
\end{align}
from the observations of \cite{CarstensenZhaiZhang2020}
with the convention $\mathrm{GLB}(j) \coloneqq 0$ for $\lambda_h(j) = \infty$.
In practice, $\alpha < 1$ is a chosen parameter and $\beta$ is an explicit constant with a positive scaling of the maximal mesh-size. Therefore, $\lambda_h(j) \leq \lambda(j)$ eventually holds whenever $\alpha + \beta \lambda_h(j) < 1$.

While discretizations into nonconvex cells are certainly allowed, the embedding constants required in GLB are rather inaccurate. Accurate guaranteed upper bounds can be computed using the HHO eigensolvers of this paper with additional computational effort.

As the focus is on GLB, the convergence of the eigensolver is not explicitly proven. However, we mention that a~priori error estimates plus elliptic regularity for the source problem allow for the spectral correctness of HHO eigensolvers \cite{CaloCicuttinDengErn2019}. The arguments of \cite{CaloCicuttinDengErn2019} are applicable to almost all examples of this paper as briefly outlined in the respective sections.

\subsection{Setting}\label{sec:abstract-setting}
For the sake of brevity, we first consider the Laplace eigenvalue problem but mention that the extension to other eigenvalue problems follow in later sections.
Let $\Omega \subset \mathbb{R}^n$ be an open bounded polyhedral Lipschitz domain in two or three space dimensions $n = 2,3$.
The Laplace eigenvalue problem seeks eigenpairs $(\lambda,u)$ with
\begin{align}\label{def:Laplace}
	-\Delta u = \lambda u \text{ in } \Omega \quad\text{and}\quad u = 0 \text{ on } \partial \Omega.
\end{align}
The variational formulation of this reads
\begin{align}\label{def:continuous-problem}
	a(u,v) = \lambda b(u,v) \quad\text{for any } v \in V
\end{align}
with $V \coloneqq H^1_0(\Omega)$, $a(u,v) = (\nabla u, \nabla v)_{L^2(\Omega)}$, and $b(u,v) = (u,v)_{L^2(\Omega)}$.
The bilinear forms $a, b$ induce the norms $\|\bullet\|_\mathrm{a}$ and $\|\bullet\|_\mathrm{b}$ in $V$.
The finite eigenvalues $0 < \lambda(1) \leq \lambda(2) \leq \dots < \infty$ of \eqref{def:continuous-problem} are given by the Rayleigh-Ritz principle
\begin{align}
	\lambda(j) = \min_{W \subset V, \mathrm{dim}(W) = j} \max_{v \in W\setminus\{0\}} \|v\|_a^2/\|v\|_b^2
\end{align}
with
the corresponding eigenvectors $u(1), u(2), \dots$, that are subject to the normalization $\|u(j)\|_b = 1$ and satisfy the orthogonality
\begin{align}\label{def:a-orthogonality}
	a(u(j), u(k)) = b(u(j), u(k)) = 0 \quad\text{for any } j \neq k.
\end{align}

\subsection{Outline} The remaining parts of this paper are organized as follows. 
\Cref{sec:laplace} introduces the HHO eigensolver for the Laplace eigenvalue problem.
\Cref{sec:leb} highlights three general conditions that lead to the GLB \eqref{ineq:lower-bound}.
The remaining sections apply the arguments of \Cref{sec:leb} to the Steklov (\Cref{sec:steklov}),
and linear elasticity (\Cref{sec:linear-elasticity}) eigenvalue problem.
\Cref{sec:compact-emb} proposes HHO eigensolvers for the computation of guaranteed bounds for constants in local embeddings and provides improvements to the GLB of earlier sections.

\subsection{Notation} Standard notation for Lebesgue and Sobolev spaces applies throughout this paper with the abbreviations $(\varphi,\psi)_{L^2(\Omega)}$ for the $L^2$ scalar product of two functions $\varphi, \psi \in L^2(\Omega)$ and $\|\bullet\| = \|\bullet\|_{L^2(\Omega)}$ denotes the $L^2$ norm.

\section{HHO eigensolver}\label{sec:laplace}
This section introduces an HHO eigensolver for the Laplace eigenvalue problem of \Cref{sec:abstract-setting}.

\subsection{Polytopal mesh}
Let $\mathcal{M}$ be a finite collection of closed \emph{convex} nonempty polyhedra with overlap of volume measure zero that cover $\overline{\Omega} = \cup_{K \in \mathcal{M}} K$.
A side $S$ of the mesh $\mathcal{M}$ is a closed connected polyhedral subset of a hyperplane $H_S$ with positive $(n-1)$-dimensional surface measure such that either (a) there exist $K_+,K_- \in \mathcal{M}$ with $S \subset H_S \cap K_+ \cap K_-$ (interior side) or (b) there exists $K_+ \in \mathcal{M}$ with $S \subset H_S \cap K_+ \cap \partial \Omega$ (boundary sides).
Let $\Sigma$ be a finite collection of sides with overlap of $(n-1)$-dimensional surface measure zero that covers the skeleton $\partial \mathcal{M} \coloneqq \cup_{K \in \mathcal{M}} \partial K = \cup_{S \in \Sigma} S$.
The normal vector $\nu_S$ of an interior side $S \in \Sigma(\Omega)$ is fixed beforehand and set $\nu_S \coloneqq \nu|_S$ for boundary sides $S \in \Sigma(\partial \Omega)$.
For $S \in \Sigma(\Omega)$, $K_+ \in \mathcal{M}$ (resp.~$K_- \in \mathcal{M}$) denotes the unique cell with $S \subset \partial K_{+}$ (resp.~$S \subset \partial K_-$) and $\nu_{K_+}|_S = \nu_S$ (resp.~$\nu_{K_-}|_S = -\nu_S$).
The jump $[v]_S$ of any function $v \in W^{1,1}(\mathrm{int}(T_+ \cup T_-);\R^m)$ along $S \in \Sigma(\Omega)$ is defined by $[v]_S \coloneqq v|_{K_+} - v|_{K_-} \in L^1(S;\R^m)$.
If $S \in \Sigma(\partial \Omega)$, then $[v]_S \coloneqq v|_S$.

The set $V_\mathrm{nc} \coloneqq H^1(\Mcal)$ is the space of all piecewise $H^1$ function with respect to the mesh $\Mcal$.
The notation $\nabla_\pw$ denotes the piecewise application of $\nabla$ without explicit reference to $\Mcal$ and analogously applies to other differential operators as well.
The seminorm $\|\bullet\|_{a_\pw}$ in $V_\mathrm{nc}$ is induced by the positive semidefinite bilinear form $a_\pw(u_\mathrm{nc}, v_\mathrm{nc}) \coloneqq (\nabla_\pw u_\mathrm{nc}, \nabla_\pw v_\mathrm{nc})_{L^2(\Omega)}$ in $V_\mathrm{nc}$.

\subsection{Finite element spaces}\label{sec:fem-spaces}
Given a subset $M \subset \R^n$ of diameter $h_M$, let $P_k(M)$ denote the space of polynomials of degree at most $k$.
For any $v \in L^1(M)$, $\Pi_M^k v \in P_k(M)$ denotes the $L^2$ projection of $v$ onto $P_k(M)$. The space of piecewise polynomials of degree at most $k$ with respect to the mesh $\mathcal{M}$ or the sides $\Sigma$ is denoted by $P_k(\mathcal{M})$ or $P_k(\Sigma)$.
Given $v \in L^1(\Omega)$ (resp.~$v \in L^1(\partial \Mcal)$), define the $L^2$ projection $\Pi_{\mathcal{M}}^k v$ (resp.~$\Pi_\Sigma^k v$) of $v$ onto $P_k(\mathcal{M})$ (resp.~$P_k(\Sigma)$) by $(\Pi_{\mathcal{M}}^k v)|_K = \Pi_K^k v|_K$ in any cell $K \in \mathcal{M}$ (resp.~$(\Pi_\Sigma^k v)|_S = \Pi_S^k v$ along any side $S \in \Sigma$).
The mesh $\mathcal{M}$ gives rise to the piecewise constant function $h_\mathcal{M} \in P_0(\mathcal{M})$ with $h_\mathcal{M}|_K = h_K$ for any $K \in \Mcal$; $h_{\max} \coloneqq \max_{K \in \mathcal{M}} h_K$ is the maximal mesh-size of $\mathcal{M}$.

\subsection{Discretization}
This subsection presents the HHO eigensolver for the computation of GLB.\\[-0.5em]

\paragraph{\emph{Discrete ansatz spaces.}}
Given a fixed $k \geq 0$, let $V_h \coloneqq P_{k+1}(\mathcal{M}) \times P_k(\Sigma(\Omega))$ denote the discrete ansatz space for $V$.
In this definition, $P_k(\Sigma(\Omega))$ is the subspace of $P_k(\Sigma)$ with the convention $v_\Sigma \in P_k(\Sigma(\Omega))$ if $v_\Sigma|_S \equiv 0$ on boundary sides $S \in \Sigma(\partial \Omega)$ to model homogenous Dirichlet boundary condition.
For any $v_h = (v_\Mcal, v_\Sigma) \in V_h$, we use the notation $v_K \coloneqq v_\Mcal|_K$ and $v_S \coloneqq v_\Sigma|_S$ to abbreviate the restriction of $v_\Mcal$ in a cell $K \in \Mcal$ and $v_\Sigma$ along a side $S \in \Sigma$.
The interpolation operator $\I_h: V \to V_h$ maps $v \in V$ to $\I_h v \coloneqq (\Pi_\Mcal^{k+1} v, \Pi_\Sigma^k v) \in V_h$.\\[-0.5em]

\paragraph{\emph{Potential reconstruction.}}
Given $v_h = (v_\Mcal, v_\Sigma) \in V_h$, the potential reconstruction $\mathcal{R}_h v_h \in W_\mathrm{nc}$ satisfies, for any $\varphi_{k+1} \in W_\mathrm{nc}$, that
\begin{align}\label{def:potential-reconstruction-1}
	a_\pw(\mathcal{R}_h v_h, \varphi_{k+1}) = - (v_\mathcal{M}, \Delta_\pw \varphi_{k+1})_{L^2(\Omega)} + \sum_{S \in \Sigma} (v_S, [\nabla_\pw \varphi_{k+1}]_S \cdot \nu_S)_{L^2(S)}.
\end{align}
This defines $\mathcal{R}_h v_h$ uniquely up to piecewise constants, which is fixed by
\begin{align}\label{def:pot-rec-2}
	\int_K \mathcal{R}_h v_h \d{x} = \int_K v_K \d{x} \quad\text{for any } K \in \Mcal.
\end{align}

\paragraph{\emph{Stabilization.}}
Given a convex polyhedra $K \in \Mcal$ of diameter $h_K$, let $x_K$ denote the midpoint of $K$. For any side $S \in \Sigma(K)$ of $K$, $K_S \coloneqq \mathrm{conv}\{x_K,S\} \subset K$ is the convex hull of $S$ and $x_K$.
We define the weight
\begin{align}\label{def:weight}
	\ell(S,K) \coloneqq |S| h_K^2/|K_S|.
\end{align}
Notice the scaling $\ell(S,K) \approx h_K$ for admissible meshes (cf.~\cite[Definition 1.38]{DiPietroErn2012} for a precise definition).
For a positive parameter $\sigma > 0$, the local stabilization $\s_K(u_h,v_h)$ reads,
for any $u_h = (u_\mathcal{M}, u_\Sigma), v_h = (v_\mathcal{M}, v_\Sigma) \in V_h$,
\begin{align}\label{def:stabilization}
	\s_K(u_h, v_h) &\coloneqq \sigma h_K^{-2}(u_K - \mathcal{R}_h u_h, v_K - \mathcal{R}_h v_h)_{L^2(K)}\nonumber\\
	&\quad + \sigma \sum_{S \in \Sigma(K)} \ell(S,K)^{-1} (\Pi_S^k(u_S - \mathcal{R}_h u_h|_K), v_S - \mathcal{R}_h v_h|_K)_{L^2(S)}
\end{align}
with $\ell(S,K)$ from \eqref{def:weight} and the global version $$\s_h(u_h, v_h) \coloneqq \sum_{K \in \mathcal{M}} \s_K(u_h,v_h).$$
A similar stabilization has been proposed (and analysed) in \cite[Example 2.8]{DiPietroDroniou2020} for the discrete ansatz space $P_{k}(\Mcal) \times P_k(\Sigma(\Omega))$, however we note that it is not sufficient for stability for the present ansatz space.\\[-0.5em]

\paragraph{\emph{Discrete eigenvalue problem.}} We seek discrete eigenpairs $(\lambda_h,u_h)$ such that 
\begin{align}\label{def:discrete-problem}
	a_h(u_h,v_h) + \s_h(u_h,v_h) = \lambda_h b_h(u_h,v_h)
\end{align}
with the bilinear forms
\begin{align}\label{def:discrete-bilinear-forms-laplace}
	a_h(u_h,v_h) \coloneqq a_\pw(\mathcal{R}_h u_h, \mathcal{R}_h v_h) \quad\text{and}\quad b_h(u_h,v_h) = (u_\Mcal, v_\Mcal)_{L^2(\Omega)}
\end{align}
for any $u_h = (u_\Mcal, u_\Sigma), v_h = (v_\Mcal, v_\Sigma) \in V_h$.
Similar to the continuous level, the discrete eigenvalues $0 \leq \lambda_h(1) \leq \dots \leq \lambda_h(N) \leq \infty$ with $N \coloneqq \mathrm{dim}\,V_h$ satisfy the Rayleigh-Ritz principle
\begin{align}\label{eq:min-max-principle-discrete}
	\lambda_h(j) = \min_{W_h \subset V_h, \mathrm{dim}(W_h) = j} \max_{v_h \in W_h \setminus\{0\}} (\|v_h\|_{a_h}^2 + \|v_h\|_{\s_h}^2)/\|v_h\|_{b_h}^2.
\end{align}

\begin{lemma}[coercivity]\label{lem:well-posedness-laplace}
	The bilinear form $a_h + \s_h$ is a scalar product in $V_h$.
	In particular, there are $\mathrm{dim}(P_{k+1}(\Mcal))$ finite discrete eigenvalues of \eqref{def:discrete-problem}.\qed
\end{lemma}
\begin{proof}
	Let $\|v_h\|_{a_h}^2 + \|v_h\|_{\s_h}^2 = 0$ for some $v_h = (v_\Mcal, v_\Sigma) \in V_h$.
	From $\|v_h\|_{a_h} = 0$, we infer 
	$\nabla_\pw \mathcal{R}_h v_h = 0$ and so, $\mathcal{R}_h v_h$ is piecewise constant. 
	Since $\|v_h\|_{\s_h} = 0$, the traces of $\mathcal{R}_h v_h$ coincide with $v_\Sigma$ and so, $v_\Mcal = \mathcal{R}_h v_h$ is a constant function with homogenous boundary condition. Therefore, $\mathcal{R}_h v_h = 0$ and $v_h = 0$. 
\end{proof}

\begin{remark}[spectral correctness]
	Quasi-best approximation estimates for the source problem are given in \cite{ErnZanotti2020,BertrandCarstensenGraessle2021} and allow for the spectral correctness of the HHO eigensolver \cite{CaloCicuttinDengErn2019}, cf.~also \cite[Theorem 4.4]{LiangTran2025}.
\end{remark}

\section{Guaranteed lower eigenvalue bounds}\label{sec:leb}
This section establishes the GLB \eqref{ineq:lower-bound}.
\subsection{Key properties}\label{sec:main-results}
First, three key properties essential for GLB are highlighted. They also hold for other eigenvalue problems as well.

Let $\mathrm{G}_h : V_\mathrm{nc} \to W_\mathrm{nc}$ denote the Galerkin projection, defined, for given $v \in V_\mathrm{nc}$, as the unique solution to
\begin{align}\label{def:Galerkin}
	a_\pw(\mathrm{G}_h v, \varphi_\mathrm{nc}) = a_\pw(v, \varphi_\mathrm{nc}) \quad\text{for any } \varphi_\mathrm{nc} \in W_\mathrm{nc}
\end{align}
with the integral mean property
\begin{align}\label{eq:Galerkin_integral_mean}
	\int_K \mathrm{G}_h v \d{x} = \int_K v \d{x} \quad\text{for any } K \in \Mcal.
\end{align}
A characteristic feature of the HHO methodology is the connection
\begin{align}\label{eq:R-I=G}
	\mathcal{R}_h \circ \I_h = \mathrm{G}_h \quad\text{in } V
\end{align}
between $\mathcal{R}_h$ and $\mathrm{G}_h$. 
Indeed, \eqref{eq:R-I=G} follows from the orthogonality $v - \mathcal{R}_h \I_h v \perp_{a_\pw} W_\mathrm{nc}$ \cite[Eq.~(18)]{DiPietroErnLemaire2014} and $\Pi_{\mathcal{M}}^0 \mathcal{R}_h \I_h v = \Pi_{\mathcal{M}}^0 v$ in \eqref{def:pot-rec-2}.
From \eqref{def:Galerkin}, we infer the orthogonality $v - \mathrm{G}_h v \perp_{a_\pw} W_\mathrm{nc}$. Together with \eqref{eq:R-I=G}, this shows
\begin{assumption}[$a$-orthogonality]\label{A}
	$\|\I_h v\|_{a_h}^2 \leq \|v\|^2_a - \|v - \mathrm{G}_h v\|_{a_\pw}^2$
\end{assumption}
\noindent  for any $v \in V$ (even with equality for the Laplace eigenvalue problem at hand).
In particular,
\begin{align}\label{ineq:v-Gv<=v}
	\|v - \mathrm{G}_h v\|_{a_\pw} \leq \|v\|_a.
\end{align}
The quasi-best approximation property
\begin{assumption}[scaling of stabilization]\label{B}
	$\|\I_h v\|_{\s_h}^2 \leq \alpha \|v - \mathrm{G}_h v\|_{a_\pw}^2$
\end{assumption}
\noindent of the stabilization with a positive constant $\alpha$ is utilized in many contributions \cite{ErnZanotti2020,BertrandCarstensenGraessle2021,CarstensenTran2024}.
An explicit upper bound for $\alpha$ can be computed using the following result.
\begin{lemma}[trace inequality]\label{lem:trace-inequality}
	Any $v \in H^1(\mathrm{int}(K))$ with $\int_K v \d{x} = 0$ satisfies
	\begin{align*}
		\sum_{S \in \Sigma(K)} \ell(S,K)^{-1}\|v\|_{L^2(S)}^2 \leq c_{\mathrm{tr}}\|\nabla v\|_{L^2(K)}^2
	\end{align*}
	with the constant $c_\mathrm{tr} \coloneqq 1/\pi^2 + 2/(n\pi)$.
\end{lemma}
\begin{proof}
	Given $S \in \Sigma(K)$, the classical trace identity implies
	\begin{align}\label{ineq:trace-inequality}
		\|v\|^2_{L^2(S)} \leq \frac{|S|}{|K_S|} \|v\|_{L^2(K_S)}^2 + \frac{2h_{K_S}|S|}{n|K_S|} \int_{K_S} |v| |\nabla v| \d{x}
	\end{align}
	cf., e.g., \cite[proof of Lemma 7.2]{Gallistl2023}.
	The sum of this over all $S \in \Sigma(K)$ with the notation $\ell(S,K)$ from \eqref{def:weight} and $h_{K_S}/h_K \leq 1$ establish
	\begin{align*}
		\sum_{S \in \Sigma(K)} \ell(S,K)^{-1}\|v\|_{L^2(S)}^2 \leq h_K^{-2} \|v\|_{L^2(K)}^2 + 2 h_K^{-1} n^{-1} \int_K |v| |\nabla v| \d{x}.
	\end{align*}
	This, the H\"older, and Poincar\'e inequality conclude the proof.
\end{proof}
\begin{remark}[simplicial cells]
	If $K$ is a simplex, then \cite[Proposition 4.3]{CarstensenZhaiZhang2020} provides an improved version of \Cref{lem:trace-inequality}.
\end{remark}
For any $v \in V$,
the best approximation property of $L^2$ projections and \eqref{eq:R-I=G} provide the upper bound
\begin{align}
	\sigma h_K^{-2} \|\Pi_K^{k+1}(v - \mathrm{G}_h v)\|_{L^2(K)}^2 + \sigma \sum_{S \in \Sigma(K)} \ell(S,K)^{-1} \|\Pi_S^k(v - \mathrm{G}_h v|_K)\|_{L^2(S)}^2\nonumber\\
	\leq \sigma h_K^{-2} \|v - \mathrm{G}_h v\|_{L^2(K)}^2 + \sigma \sum_{S \in \Sigma(K)} \ell(S,K)^{-1} \|v - \mathrm{G}_h v|_K\|_{L^2(S)}^2
	\label{ineq:proof-stabilization}
\end{align}
for the stabilization $\s_K(\I_h v, \I_h v)$. This, the Poincar\'e, and trace inequality from \Cref{lem:trace-inequality} establish \ref{B} with the constant
\begin{align*}
	\alpha \coloneqq \sigma/\pi^2 + \sigma c_\mathrm{tr}.
\end{align*}
Note that $\alpha$ is independent of the polynomial degree $k$.
The final ingredient towards GLB involves the bound $\|v - \mathrm{G}_h v\|_{L^2(\Omega)}^2 \leq \beta\|v - \mathrm{G}_h v\|_{a_\pw}^2$ with $\beta \leq h_\mathrm{max}^2/\pi^2$ from a Poincar\'e inequality \cite{Bebendorf2003}.
This, the Pythagoras theorem $\|\I_h v\|^2_{b_h} = \|v\|_b^2 - \|v - \Pi_{\mathcal{M}}^{k+1} v\|_{L^2(\Omega)}^2$, and $\|v - \Pi_{\mathcal{M}}^{k+1} v\|_{L^2(\Omega)} \leq \|v - \mathrm{G}_h v\|_{L^2(\Omega)}$ allow for
\begin{assumption}[$b$-orthogonality]\label{C}
	$\|\I_h v\|_{b_h}^2 \geq \|v\|_b^2 - \beta \|v - \mathrm{G}_h v\|_{a_\pw}^2$.
\end{assumption}
\subsection{Main result}
We recall the result mentioned in the introduction.
\begin{theorem}[GLB for Laplace]\label{thm:dLEB}
	The bound \eqref{ineq:lower-bound} holds for any $j = 1, \dots, N$ with the constants $\alpha \coloneqq \sigma/\pi^2 + \sigma c_\mathrm{tr}$ and $\beta \coloneqq h_{\max}^2/\pi^2$. In particular,
	$$\mathrm{GLB}(j) \coloneqq \min\{1, 1/(\sigma/\pi^2 + \sigma c_\mathrm{tr} + h_\mathrm{max}^2 \lambda_h(j)/\pi^2)\}\lambda_h(j) \leq \lambda(j).$$
\end{theorem}
The proofs in this section follow \cite{CarstensenZhaiZhang2020} and are outlined for the sake of completeness. The following result is required for the proof of \Cref{thm:dLEB}.
Define the finite dimensional subspace $W \coloneqq \mathrm{span}\{u(1), \dots, u(j)\}$ of $V$.
\begin{lemma}[linear independency]\label{lem:linear-independency}
	If $\beta \lambda(j) < 1$, then (a)--(b) hold.
	\begin{enumerate}[wide]
		\item[(a)] The interpolations $\I_h u(1)$, $\dots$, $\I_h u(j)$ of the exact eigenvectors $u(1), \dots, u(j)$ are linearly independent in $V_h$.
		\item[(b)] Suppose additionally $\lambda_h(j) < \infty$, then there exists $v \in W$ with $\|v\|_b = 1$ and
		\begin{align*}
			(1 - \alpha - \beta \lambda_h(j))\|v - \mathrm{G}_h v\|_{a_\pw}^2 \leq \lambda(j) - \lambda_h(j).
		\end{align*}
	\end{enumerate}
\end{lemma}
\begin{proof}[Proof of (a)]
	Suppose otherwise, then there are real numbers $t_1, \dots, t_j$ such that $v \coloneqq t_1 u(1) + \dots + t_ju(j)$ satisfies $\|v\|_b = 1$ but $\I_h v = 0$.
	Elementary algebra with the orthogonality in \eqref{def:a-orthogonality} proves
	$1 = \|v\|_b^2 = t_1^2 \|u(1)\|_b^2 + \dots t_j^2 \|u(j)\|_b = t_1^2 + \dots + t_j^2$ and therefore,
	\begin{align}\label{ineq:bound-a-v-v}
		\|v\|_a^2 = t_1^2 \|u(1)\|^2_a + \dots + t_j^2 \|u(j)\|_a^2 &= t_1^2 \lambda(1) + \dots + t_j^2 \lambda(j)\nonumber\\
		&\leq (t_1^2 + \dots + t_j^2)\lambda(j) = \lambda(j).
	\end{align}
	The combination of this with \Cref{C} and \eqref{ineq:v-Gv<=v} results in
	\begin{align}\label{ineq:proof-dLEB-claim-1}
		1 = \|v\|_b^2 \leq \beta \|v - \mathrm{G}_h v\|_{a_\pw}^2 \leq \beta \|v\|^2_a \leq \beta \lambda(j);
	\end{align}
	a contradiction to $\beta\lambda(j) < 1$.\\[-0.5em]
	
	\paragraph{\emph{Proof of (b)}}
	Since the interpolations $\I_h u(1), \dots, \I_h u(j)$ are linearly independent from (a), $\I_h W \subset V_h$ is a $j$-dimensional subspace of $V_h$.
	The Rayleigh-Ritz principle \eqref{eq:min-max-principle-discrete} on the discrete level shows
	\begin{align*}
		\lambda_h(j) \leq \max_{v_h \in \I_h W\setminus\{0\}} (\|v_h\|_{a_h}^2 + \|v_h\|_{\s_h}^2)/\|v_h\|_{b_h}^2.
	\end{align*}
	Let $v \in W \setminus \{0\}$ with $\|v\|_b = 1$ be given such that $\I_h v$ is a maximizer of this Rayleigh quotient (which can attain the value $+\infty$).
	Then
	\begin{align}\label{ineq:proof-dLEB-min-max}
		\lambda_h(j) \|\I_h v\|_{b_h}^2 \leq \|\I_h v\|_{a_h}^2 + \|\I_h v\|_{\s_h}^2.
	\end{align}
	From \Cref{A} and $\|v\|_a^2 \leq \lambda(j)$ in \eqref{ineq:bound-a-v-v}, we deduce that
	$\|\I_h v\|_{a_h}^2 \leq \lambda(j) - \|v - \mathrm{G}_h v\|_{a_\pw}^2$. This and \Cref{B} imply that the right-hand side of \eqref{ineq:proof-dLEB-min-max} is bounded by $\lambda(j) - (1 - \alpha)\|v - \mathrm{G}_h v\|_{a_\pw}^2$, while $\lambda_h(j)(1 - \beta \|v - \mathrm{G}_h v\|_{a_\pw}^2)$ is a lower bound for the left-hand side of \eqref{ineq:proof-dLEB-min-max} thanks to \Cref{C} and $\|v\|_b = 1$. Hence,
	\begin{align*}
		\lambda_h(j)(1 - \beta \|v - \mathrm{G}_h v\|_{a_\pw}^2) \leq \lambda(j) - (1 - \alpha)\|v - \mathrm{G}_h v\|_{a_\pw}^2.
	\end{align*}
	Reorganizing these terms concludes the proof of (b).
\end{proof}

\begin{proof}[Proof of \Cref{thm:dLEB}]
	Fix $1 \leq j \leq N$.
	Recall the convention $\mathrm{GLB}(j) \coloneqq 0$ for $\lambda_h(j) = \infty$, which is a trivial bound for $\lambda(j)$. Hence, throughout this proof, we assume that $\lambda_h(j) < \infty$.
	The proof 
	consists of the two cases (a) $\alpha + \beta \lambda_h(j) \leq 1$ and (b) $\alpha + \beta \lambda_h(j) > 1$.\\[-0.5em]
	
	\noindent\textbf{(a)} Suppose that $\alpha + \beta \lambda_h(j) \leq 1$, then the assertion reads
	$\lambda_h(j) \leq \lambda(j)$.
	If $\beta \lambda(j) \geq 1$, then $\beta \lambda_h(j) \leq \alpha + \beta \lambda_h(j) \leq 1 \leq \beta \lambda(j)$, whence $\lambda_h(j) \leq \lambda(j)$.
	If $\beta \lambda(j) < 1$, then \Cref{lem:linear-independency}(b) and $1 - \alpha - \beta \lambda_h(j) \geq 0$ imply $\lambda_h(j) \leq \lambda(j)$.\\[-0.5em]
	
	\noindent\textbf{(b)} Suppose that $\alpha + \beta \lambda_h(j) > 1$, then the claim reads $\lambda_h(j)/(\alpha + \beta \lambda_h(j)) \leq \lambda(j)$.
	If $\beta \lambda(j) < 1$, then \Cref{lem:linear-independency}(b) implies the existence of a $v \in W$ with $\|v\|_b = 1$ and $(1 - \alpha - \beta \lambda_h(j))\|v - \mathrm{G}_h v\|_{a_\pw}^2 \leq \lambda(j) - \lambda_h(j)$. The normalization $\|v\|_b = 1$ provides $\|v\|_a^2 \leq \lambda(j)$ in \eqref{ineq:bound-a-v-v} and so, $\|v - \mathrm{G}_h v\|_{a_\pw}^2 \leq \|v\|_a^2 \leq \lambda(j)$ from \eqref{ineq:v-Gv<=v}. This and $1 - \alpha - \beta \lambda_h(j) < 0$ yield
	\begin{align*}
		 (1 - \alpha - \beta \lambda_h(j)) \lambda(j) \leq (1 - \alpha - \beta \lambda_h(j))\|v - \mathrm{G}_h v\|_{a_\pw}^2 \leq \lambda(j) - \lambda_h(j),
	\end{align*}
	whence $\lambda_h(j)/(\alpha + \beta \lambda_h(j)) \leq \lambda(j)$. If $1 \leq \beta \lambda(j)$, then $\lambda_h(j)/(\alpha + \beta \lambda_h(j)) \leq 1/\beta \leq \lambda(j)$. This concludes the proof.
\end{proof}

As in \cite{CarstensenZhaiZhang2020,CarstensenErnPuttkammer2021,BertrandCarstensenGraessle2021}, the conditions \Cref{A}--\Cref{C} allow for GLB from a~priori information on the exact eigenvalue.

\begin{corollary}[a~priori GLB]\label{cor:a-priori}
	If $\alpha + \beta \lambda(j) \leq 1$ and $\lambda_h(j) < \infty$, then $\lambda_h(j) \leq \lambda(j)$.
\end{corollary}
\begin{proof}
	The bound $\alpha + \beta \lambda(j) \leq 1$ implies
	$1 - \alpha - \beta \lambda_h(j) \geq \alpha + \beta\lambda(j) - \alpha - \beta\lambda_h(j) \geq \beta(\lambda(j) - \lambda_h(j))$.
	Since $\beta \lambda(j) < 1$, \Cref{lem:linear-independency}(b) holds and therefore,
	\begin{align}\label{ineq:proof-a-priori-LEB}
		\beta(\lambda(j) - \lambda_h(j)) \|v - \mathrm{G}_h v\|_{a_\pw}^2 \leq \lambda(j) - \lambda_h(j)
	\end{align}
	for some $v \in V$ with $\|v\|_b = 1$.
	Recall $\beta \|v - \mathrm{G}_h v\|_{a_\pw}^2 \leq \beta \|v\|_a^2 \leq \beta \lambda(j)$
	from \eqref{ineq:v-Gv<=v} and \eqref{ineq:bound-a-v-v}.
	The combination of this with \eqref{ineq:proof-a-priori-LEB} results in
	\begin{align*}
		\beta \lambda(j)(\lambda(j) - \lambda_h(j)) \leq \lambda(j) - \lambda_h(j).
	\end{align*}
	This inequality can only hold if $\lambda_h(j) \leq \lambda(j)$ because $\beta \lambda(j) < 1$.
\end{proof}
\begin{remark}[generalization]
	Notice that the proof of \Cref{thm:dLEB} only relies on the properties \ref{A}--\ref{C}. Thus, an extension of \Cref{thm:dLEB} to other eigenvalue problems is straightforward by verifying these conditions in \Cref{sec:steklov}--\ref{sec:compact-emb} below.
\end{remark}
\subsection{Numerical example}
The computer experiments throughout this paper are carried out on regular triangulation into simplices.
Some general remarks on the realization precede the numerical results.

\subsubsection{Displayed quantities}
The convergence history plots display the error $\lambda_\mathrm{C}(j) - \mathrm{GLB}(j)$ against the number of degrees of freedom $\mathrm{ndof}$ in a log-log plot, where the upper bound $\lambda_\mathrm{C}(j)$ is the Rayleigh quotient of some conforming function $u_\mathrm{C}(j) \in P_{k+1}(\Mcal) \cap V$.
If $j = 1$, then $u_\mathrm{C}(j)$ is the scaled nodal average of $\mathcal{R}_h u_h(j)$ with $\|u_\mathrm{C}\|_b = 1$
and for $j \geq 1$, $u_\mathrm{C}(j)$ can be obtained from 
the conforming Galerkin $P_{k+1}$ FEM on the same mesh.
At least on simplicial meshes (without hanging nodes), $\|\I_h v_\mathrm{C}\|_{a_h}^2 + \|\I_h v_\mathrm{C}\|_{\s_h}^2 = \|v_\mathrm{C}\|_a^2$ and $\|\I_h v_\mathrm{C}\|_{b_h}^2 = \|v_\mathrm{C}\|_b^2$ holds for any $v_\mathrm{C} \in P_{k+1}(\Mcal) \cap V$. Therefore, the discrete Rayleigh-Ritz principle \eqref{eq:min-max-principle-discrete} shows $\lambda_h(j) \leq \lambda_\mathrm{C}(j)$.
In combination with \Cref{thm:dLEB}, $\mathrm{GLB}(j) \leq \{\lambda, \lambda_h(j)\} \leq \lambda_\mathrm{C}(j)$ and so, $|\lambda(j) - \lambda_h(j)| \leq \lambda_\mathrm{C}(j) - \mathrm{GLB}(j)$.

\subsubsection{Adaptive refinement indicator}
If $\mathrm{GLB}(j) = \lambda_h(j)$, the error estimator $\lambda_\mathrm{C}(j) - \lambda_h(j)$ can be localized following \cite[Lemma 6.3]{StrangFix1973}. Let $u_\mathrm{C}(j)$ be given such that $\lambda_\mathrm{C}(j) = \|u_\mathrm{C}(j)\|_a^2$ with the normalization $\|u_\mathrm{C}(j)\|_b^2 = 1$ and $(u_\mathrm{C}(j), u_\Mcal(j))_{L^2(\Omega)} > 0$, where $u_h(j) = (u_\Mcal(j),u_\Sigma(j)) \in V_h$ is the computed eigenvector of $\lambda_h(j)$.
Elementary algebra with $\|\mathcal{R}_h u_h(j)\|^2_{a_\pw} = \lambda_h(j) - \|u_h(j)\|_{\s_h}^2$ shows
\begin{align*}
	&\lambda_\mathrm{C}(j) - \lambda_h(j) = \|u_\mathrm{C}(j)\|_a^2 - \lambda_h(j)\\
	&\quad= \|u_\mathrm{C}(j) - \mathcal{R}_h u_h(j)\|^2_{a_\pw} + 2a_\pw(u_\mathrm{C}(j), \mathcal{R}_h u_h(j)) - \|\mathcal{R}_h u_h(j)\|^2_{a_\pw} - \lambda_h(j)\\
	&\quad= \|u_\mathrm{C}(j) - \mathcal{R}_h u_h(j)\|^2_{a_\pw} + 2a_\pw(u_\mathrm{C}(j), \mathcal{R}_h u_h(j)) - 2 \lambda_h(j) + \|u_h(j)\|_{\s_h}^2.
\end{align*}
The discrete variational formulation \eqref{def:discrete-problem} and the commuting property \eqref{eq:R-I=G} show
\begin{align*}
	2a_\pw(u_\mathrm{C}(j), \mathcal{R}_h u_h(j)) &= 2\lambda_h(u_\Mcal, u_\mathrm{C}(j))_{L^2(\Omega)}\\
	& = -\lambda_h\|u_\mathrm{C}(j) - u_\Mcal(j)\|^2_{L^2(\Omega)} + \lambda_h(\|u_\mathrm{C}(j)\|^2_b + \|u_h(j)\|_{b_h}^2).
\end{align*}
The two previously displayed formula with $\|u_\mathrm{C}(j)\|^2_b + \|u_h(j)\|_{b_h}^2 = 2$ conclude
\begin{align}\label{def:eta}
	\lambda_\mathrm{C}(j) - \lambda_h(j) \leq \eta \coloneqq \|u_\mathrm{C}(j) - \mathcal{R}_h u_h(j)\|^2_{a_\pw} + \|u_h(j)\|^2_{\s_h}.
\end{align}
Notice that the term $-\lambda_h\|u_\mathrm{C}(j) - u_\Mcal(j)\|^2_{L^2(\Omega)} \leq 0$ is neglected in the last formula. However, it is expected that $\|u_\mathrm{C}(j) - u_\Mcal(j)\|^2_{L^2(\Omega)}$ is of higher order due to additional smoothness of the exact eigenvector.
A generalization of \eqref{def:eta} to general polyhedral meshes is possible with the addition of $\|\I_h u_\mathrm{C}\|_{\s_h}^2$ on the right-hand side of \eqref{def:eta}.
The bound \eqref{def:eta}
motivates the local refinement indicator
\begin{align*}
	\eta(K) \coloneqq \|\nabla_\pw (u_\mathrm{C}(j) - \mathcal{R}_h u_h(j))\|^2_{L^2(K)} + \s_K(u_h(j), u_h(j)).
\end{align*}
This localization is relevant for simple eigenvalues. For the approximation of eigenvalue clusters, we suggest the use of residual-based error estimators, e.g., from \cite{CarstensenGraessleTran2024}, as in \cite{Gallistl2014}.

\subsubsection{Adaptive algorithm}
Whenever $\alpha + \beta \lambda_h(j) > 1$, the mesh is refined uniformly.
Otherwise,
the refinement indicator $\eta(K)$ is utilized in the standard D\"orfler marking with bulk parameter $1/2$, i.e., at each refinement step, a subset $\mathfrak{M} \subset \Mcal$ of minimal cardinality is selected so that
\begin{align*}
	\eta \leq \frac{1}{2}\sum\nolimits_{K \in \mathfrak{M}} \eta(K).
\end{align*}

\begin{figure}[ht]
	\begin{minipage}[b]{0.475\textwidth}
		\centering
		\includegraphics[height=5cm]{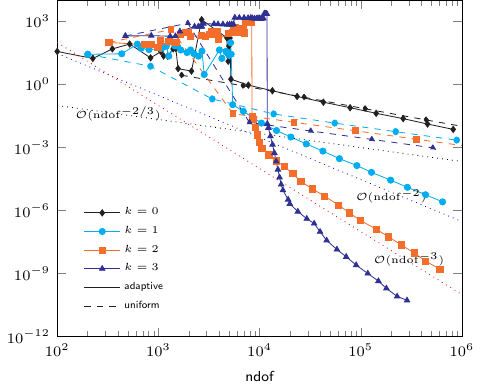}
	\end{minipage}\hfill
	\begin{minipage}[t]{0.475\textwidth}
		\centering
		\includegraphics[height=4.6cm]{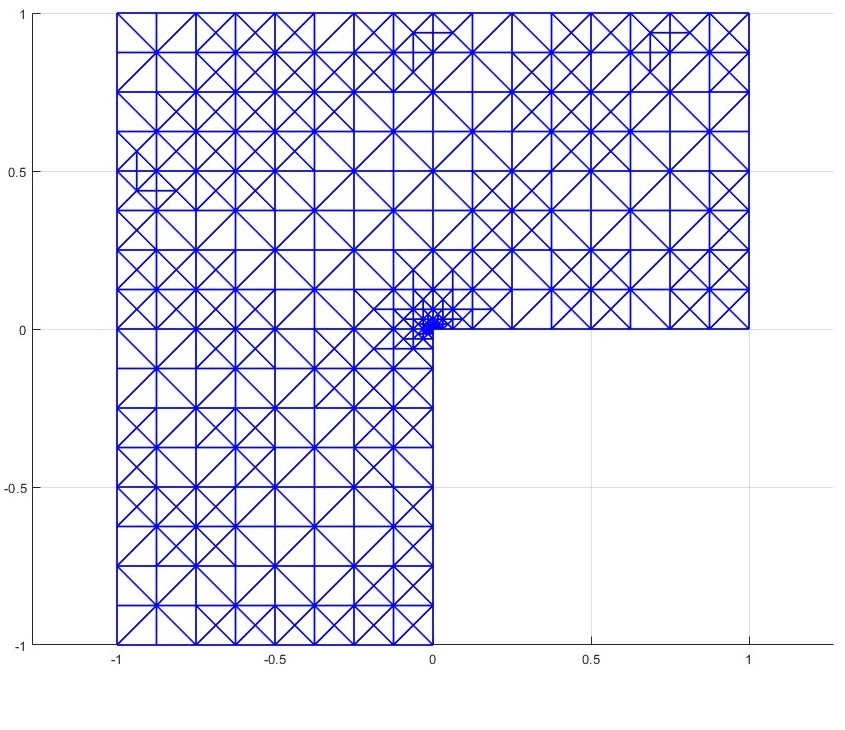}
	\end{minipage}
	\captionsetup{width=1\linewidth}
	\caption{(a) Convergence history plot of $\lambda_\mathrm{C}(1) - \mathrm{GLB}(1)$ and (b) adaptive mesh with 736 triangles ($k = 2$) for the Laplace eigenvalue problem in \Cref{sec:num_ex_laplace}}
	\label{fig:Laplace}
\end{figure}

\subsubsection{Computer benchmark}\label{sec:num_ex_laplace}
This benchmark approximates the first Laplace eigenvalue in the $L$-shaped domain $\Omega \coloneqq (-1,1) \setminus ([0,1] \times [-1,0])$ with the reference value $\lambda(1) = 9.63972384$ from the bound $\mathrm{GLB}(1) \leq \lambda(1) \leq \lambda_\mathrm{C}(1)$.
We chose $\sigma \coloneqq 2^{-1}(\pi^{-2} + c_\mathrm{tr})^{-1} = 0.9598$ in \Cref{thm:dLEB}, which provides $\mathrm{GLB}(j) = \min\{1, 1/(1/2 + h_\mathrm{max}^2 \lambda_h(j)/\pi^2)\}\lambda_h(j) \leq \lambda(j)$.
The observations in this example also applies to all other benchmarks in \Cref{sec:steklov}--\ref{sec:linear-elasticity} below.
\Cref{fig:Laplace}(a) displays the convergence history of $\lambda_\mathrm{C}(1) - \mathrm{GLB}(1)$.
Uniform mesh refinements lead to the suboptimal convergence rates $2/3$ due to the expected singularity of the first eigenvector. Adaptive mesh computation refines towards the origin as shown in \Cref{fig:Laplace}(b) and recovers the optimal convergence rates $k+1$ for all displayed polynomial degrees $k$.
Undisplayed numerical experiments show that there are only marginal differences between $\eta$ and $\lambda_\mathrm{C}(1) - \mathrm{GLB}(1)$ except on very coarse meshes. 
Notice that \Cref{fig:Laplace}(a) also displays a disadvantage of this method in comparison to other numerical schemes \cite{LiuOishi2013,CarstensenGedicke2014,Gallistl2023} with postprocessed eigenvalue bounds.
The constants in local compact embeddings (e.g., trace inequalities) provide an upper bound for $\sigma$ so that $\alpha < 1$. Since overestimation of these constants is expected, $\sigma$ may be rather small in comparison to the best possible value with direct impact on the numerical method. We observe a preasymptotic range, where the discrete eigenvalue and lower bound are close to zero (and thus, not relevant) until the stabilization is resolved.
A remedy is the computation of accurate upper bounds of the involved constants; we refer to \Cref{sec:compact-emb} for further details.

\section{Steklov eigenvalue problem}\label{sec:steklov}
The Steklov eigenvalue problem seeks eigenpairs $(\lambda, u)$ with
\begin{align*}
	- \Delta u + u = 0 \text{ in } \Omega \quad\text{and}\quad \partial u/\partial \nu = \lambda u \text{ on } \partial \Omega. 
\end{align*}
The weak formulation of this is \eqref{def:continuous-problem} with $V \coloneqq H^1(\Omega)$, $a(u,v) \coloneqq (\nabla u, \nabla v)_{L^2(\Omega)} + (u, v)_{L^2(\Omega)}$, and $b(u,v) \coloneqq (u, v)_{L^2(\partial \Omega)}$.

\subsection{HHO eigensolver}
Let $V_\mathrm{nc} \coloneqq H^1(\Mcal)$, $a_\pw \coloneqq (\nabla_\pw \bullet, \nabla_\pw \bullet)_{L^2(\Omega)}$, and $W_\mathrm{nc}(\Mcal) \coloneqq P_{k+1}(\Mcal)$ with $k \geq 0$.
We utilize the discrete ansatz space $V_h \coloneqq P_k(\Mcal) \times P_{k+1}(\Sigma)$ with the interpolation $\I_h v \coloneqq (\Pi_\Mcal^k v, \Pi_\Sigma^{k+1} v) \in V_h$ for any $v \in V$.
The discrete problem \eqref{def:discrete-problem} is defined with the bilinear forms
\begin{align*}
	a_h (u_h, v_h) &\coloneqq (\nabla_\pw \mathcal{R}_h u_h, \nabla_\pw \mathcal{R}_h v_h)_{L^2(\Omega)} + (u_\Mcal, v_\Mcal)_{L^2(\Omega)},\\
	b_h(u_h,v_h) &\coloneqq (u_\Sigma, v_\Sigma)_{L^2(\partial \Omega)},
\end{align*}
and the stabilization $\s_h(u_h,v_h) \coloneqq \sum_{K \in \mathcal{M}} \s_K(u_h,v_h)$,
\begin{align*}
	\s_K(u_h,v_h) &\coloneqq \sigma h_K^{-2}(\Pi_K^k(u_K - \mathcal{R}_h u_h), v_K - \mathcal{R}_h v_h)_{L^2(\Omega)}\\
	&\quad + \sigma \sum_{S \in \Sigma(K)} \ell(S,K)^{-1}(u_S - \mathcal{R}_h u_h|_K, v_S - \mathcal{R}_h v_h|_K)_{L^2(S)},
\end{align*}
for any $u_h = (u_\Mcal, u_\Sigma), v_h = (v_\Mcal, v_\Sigma) \in V_h$, where the potential reconstruction $\mathcal{R}_h$ is defined verbatim as in \eqref{def:potential-reconstruction-1}--\eqref{def:pot-rec-2}. It is straightforward to verify the coercivity of $a_h + \s_h$ as in the proof of \Cref{lem:well-posedness-laplace}.

\begin{lemma}[coercivity]
	The bilinear form $a_h + \s_h$ is a scalar product in $V_h$.
	In particular, there are $\mathrm{dim}(P_{k+1}(\Sigma(\partial \Omega)))$ finite discrete eigenvalues of \eqref{def:discrete-problem}.\qed
\end{lemma}

\begin{remark}[spectral correctness]
	An a~priori error analysis of HHO eigensolvers for the Steklov eigenvalue problem has been established in \cite{BustinzaCicuttinLombardi2025}. The arguments therein apply to the present HHO eigensolver and lead to its spectral correctness. 
\end{remark}

\subsection{Lower eigenvalue bounds}
This subsection verifies \Cref{A}--\Cref{C} and derives GLB for the Steklov eigenvalue problem.
The Galerkin projection is defined by \eqref{def:Galerkin}--\eqref{eq:Galerkin_integral_mean}.
We utilize the following trace inequality. Given a boundary side $S \in \Sigma(\partial \Omega)$,
there exists a unique cell $K \in \Mcal$ with $S \in \Sigma(K)$.
The convex hull of $S$ and a point $x \in K$ of largest volume is denoted by $K^\mathrm{m}_S$.
(Notice that $K_S^\mathrm{m} = K$ for simplicial meshes.)
The trace inequality \eqref{ineq:trace-inequality} in $K_S^\mathrm{m}$ and a H\"older inequality provide, for any $v \in H^1(K)$, that
\begin{align}\label{ineq:trace-2}
	\|v\|_{L^2(S)}^2 \leq \frac{|S|}{|K_S^\mathrm{m}|} \|v\|_{L^2(K)}^2 + \frac{2 h_{K_S^\mathrm{m}}|S|}{n|K_S^\mathrm{m}|}\|v\|_{L^2(K)} \|\nabla v\|_{L^2(K)}.
\end{align}
This motivates the definition of the constant
\begin{align}\label{def:beta-st}
	\beta_\mathrm{St} \coloneqq c_\mathrm{0} \max_{S \in \Sigma(\partial \Omega)} \frac{h_K|S|}{|K_S^\mathrm{m}|}(h_K/\pi^2 + 2h_{K_S^\mathrm{m}}/(n\pi))
\end{align}
utilized in \Cref{cor:Steklov} below, where $c_0$ is the maximal number of boundary sides belonging to the same cell.

\begin{corollary}[GLB for Steklov]\label{cor:Steklov}
	It holds
	\begin{align}
		\mathrm{GLB}(j) \coloneqq \min\{1, 1/(\sigma/\pi^2 + \sigma c_\mathrm{tr} + \beta_\mathrm{St} \lambda_h(j))\} \leq \lambda(j).
	\end{align}
\end{corollary}

\begin{proof}
	The property \eqref{eq:R-I=G} is independent of boundary data and holds here as well. This, the Pythagoras theorem with the orthogonality $v - \mathrm{G}_h v \perp_{a_\pw} W_\mathrm{nc}$, and the best approximation property of $\Pi_\Mcal^k$ imply, for any $v \in V$, that
	\begin{align*}
		\|\I_h v\|_{a_h}^2 &= \|\mathrm{G}_h v\|_{a_\pw}^2 + \|\Pi_\Mcal^k v\|^2\\
		& = \|v\|_{a_\pw}^2 - \|v - \mathrm{G}_h v\|_{a_\pw}^2 + \|\Pi_\Mcal^k v\|^2 \leq \|v\|_a^2 - \|v - \mathrm{G}_h v\|_{a_\pw}^2,
	\end{align*}
	which is \Cref{A}. The same arguments as in \eqref{ineq:proof-stabilization} provide \Cref{B} with $\alpha \coloneqq \sigma/\pi^2 + \sigma c_\mathrm{tr}$.
	For any boundary side $S \in \partial \Omega$, let $K$ denote the unique cell with $S \in \Sigma(K)$. The Pythagoras theorem and the best approximation property of $L^2$ projections imply
	\begin{align*}
		\|\Pi_S^{k+1} v\|^2_{L^2(S)} = \|v\|_{L^2(S)}^2 - \|v - \Pi_S^{k+1} v\|_{L^2(S)}^2 \leq \|v\|_{L^2(S)}^2 - \|v - \mathrm{G}_h v\|_{L^2(S)}^2.
	\end{align*}
	Invoking the trace inequality \eqref{ineq:trace-2} and the Poincar\'e inequality, we infer that
	\begin{align}\label{ineq:trace-steklov}
		\|v - \mathrm{G}_h v\|_{L^2(S)}^2 \leq \frac{h_K|S|}{|K_S^\mathrm{m}|}(h_K/\pi^2 + 2 h_{K_S^\mathrm{m}}/(n \pi))\|\nabla (v - \mathrm{G}_h v)\|_{L^2(K)}^2.
	\end{align}
	The combination of the two previously displayed formula with an overlap argument results
	in
	\begin{align*}
		\|\I_h v\|_{b_h}^2 = \|\Pi_\Sigma^{k+1} v\|_{L^2(\partial \Omega)}^2 \leq \|v\|_b^2 - \beta_\mathrm{St}\|v - \mathrm{G}_h v\|_{a_\pw}^2,
	\end{align*}
	whence \Cref{C} holds with $\beta \coloneqq \beta_\mathrm{St}$.
	Note that, in the proof of \Cref{lem:linear-independency}, the normalization $\|v\|_b = 1$ is still possible even if $b$ is merely a seminorm because $b$ is a norm in $\mathrm{span}\{u(1), \dots, u(j)\}$. Therefore, the proof of \Cref{thm:dLEB} applies verbatim and \eqref{ineq:lower-bound} holds with the aforementioned constants.
\end{proof}

\begin{figure}[ht]
	\centering
	\includegraphics[height=8cm]{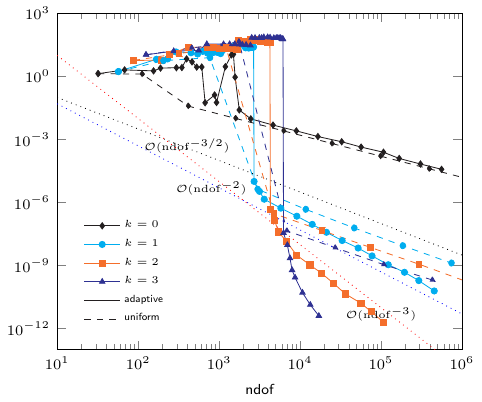}
	\captionsetup{width=1\linewidth}
	\caption{Convergence history plot of $\lambda_\mathrm{C}(1) - \mathrm{GLB}(1)$ for the Steklov eigenvalue problem in \Cref{sec:num-ex-steklov}}
	\label{fig:Steklov}
\end{figure}

\subsection{Computer benchmark}\label{sec:num-ex-steklov}
This benchmark approximates the first Steklov eigenvalue in the $L$-shaped domain $\Omega \coloneqq (-1,1) \setminus ([0,1] \times [-1,0])$ with the reference value
$\lambda(1) = 0.34141604251$ from the bound $\mathrm{GLB}(1) \leq \lambda(1) \leq \lambda_\mathrm{C}(1)$. We set $\sigma \coloneqq 2^{-1}(\pi^{-2} + c_\mathrm{tr})^{-1} = 0.9598$ so that $\mathrm{GLB}(1) = \min\{1,1/(1/2 + \beta_\mathrm{st}\lambda_h(1))\}\lambda_h(1) \leq \lambda(1)$ from \Cref{cor:Steklov}. The computation of $\lambda_\mathrm{C}(1)$ and $u_\mathrm{C}(1)$ follows \Cref{sec:num_ex_laplace}. For the first discrete eigenvalue $u_h(1) = (u_\Mcal(1), u_\Sigma(1))$, the bound
\begin{align*}
	\lambda_\mathrm{C}(1) - \lambda_h(1) \leq \|u_\mathrm{C}(1) - \mathcal{R}_h u_h(1)\|^2_{a_\pw} + \|u_\mathrm{C}(1) - u_\Mcal(1)\|^2_{L^2(\Omega)} + \|u_h(1)\|_{\s_h}^2
\end{align*}
similar to \eqref{def:eta} provides a refinement indicator by localization of the right-hand side. \Cref{fig:Steklov} displays the convergence history of $\lambda_\mathrm{C}(1) - \lambda_h(1)$ and the observations in \Cref{sec:num_ex_laplace} apply.
An advantage of this method over \cite{YouXieLiu2019,Gallistl2023} is that, even for smooth eigenfunctions, the lower bounds therein can only converge towards the exact eigenvalue with the rate $1/2$ due to the scaling of trace inequalities.

\section{Linear elasticity}\label{sec:linear-elasticity}
Let the boundary $\partial \Omega$ of $\Omega$ be split into a closed connected nonempty Dirichlet part $\Gamma_\mathrm{D} \subset \partial \Omega$ and a Neumann part $\Gamma_\mathrm{N} \coloneqq \partial \Omega \setminus \Gamma_\mathrm{D}$.
The linear elasticity eigenvalue problem seeks eigenpairs $(\lambda,u)$ such that
\begin{align}\label{def:PDE-elasticity}
	-\div\,\sigma = \lambda u \text{ in } \Omega, \quad \sigma = \mathbb{C} \varepsilon (u), \quad u = 0 \text{ on } \Gamma_\mathrm{D}, \quad\text{and } \sigma \nu = 0 \text{ on } \Gamma_\mathrm{N}.
\end{align}
Here, $\varepsilon(v) \coloneqq \mathrm{sym}(\D v)$ is the symmetric part of the gradient $\D v$ of $v$ and $\mathbb{C} A \coloneqq 2 \mu A + \kappa \mathrm{tr}(A) \mathrm{I}_n$ for any $A \in \mathbb{S}$ denotes the linearized strain tensor with Lam\'e parameters $\lambda, \mu > 0$.
The weak formulation of \eqref{def:PDE-elasticity} is \eqref{def:continuous-problem} with $V \coloneqq \{v \in H^1(\Omega)^n : v = 0 \text{ on } \Gamma_\mathrm{D}\}$, $a(u,v) \coloneqq (\mathbb{C} \varepsilon(u), \varepsilon(v))_{L^2(\Omega)}$, and $b(u,v) \coloneqq (u,v)_{L^2(\Omega)}$.

\subsection{HHO eigensolver}\label{sec:le-eigensolver}
Assume that the Dirichlet boundary can be exactly resolved by the mesh $\Mcal$.
Let $V_\mathrm{nc} \coloneqq H^1(\Mcal)^n$, $a_\pw \coloneqq (\mathbb{C} \varepsilon_\pw(\bullet), \varepsilon_\pw(\bullet))_{L^2(\Omega)}$, and $W_\mathrm{nc} \coloneqq P_{k+1}(\Mcal)^n$.
Given any $k \geq 1$, $V_h \coloneqq P_{k+1}(\Mcal)^n \times P_k(\Mcal)^n$ denotes the discrete ansatz space. The potential reconstruction $\mathcal{R}_h v_h \in W_\mathrm{nc}$ of $v_h = (v_\Mcal, v_\Sigma) \in V_h$ satisfies, for any $\varphi_\mathrm{nc} \in W_\mathrm{nc}$, that
\begin{align}\label{def:pot-rec-el-1}
	a_\pw(\mathcal{R}_h v_h, \varphi_\mathrm{nc}) &= -(v_\Mcal, \div_\pw \mathbb{C} \varepsilon_\pw(\varphi_\mathrm{nc}))_{L^2(\Omega)}\nonumber\\
	&\qquad+ \sum_{S \in \Sigma} (v_S, [\mathbb{C} \varepsilon_\pw (\varphi_\mathrm{nc})]_S \nu_S)_{L^2(S)}.
\end{align}
This defines $\mathcal{R}_h u_h$ uniquely up to the degrees of freedom associated with rigid body motions, which are fixed, for any $K \in \Mcal$, by
\begin{align}\label{def:pot-rec-el-2}
	\int_K \mathcal{R}_h v_h \d{x} = \int_K v_K \d{x}, \int_K \D_\mathrm{ss} \mathcal{R}_h v_h \d{x} = \sum_{S \in \Sigma(K)} \int_S \mathrm{asym}(\nu_K \otimes v_S) \d{s},
\end{align}
where $\D_\mathrm{ss}$ denotes the asymmetric part of the gradient and $\nu_K \otimes v_S \coloneqq \nu_K v_S^t \in \mathbb{R}^{n \times n}$.
By design, $\mathcal{R}_h$ satisfies the following.

\begin{lemma}[projection property]\label{lem:projection-linear-elasticity}
	Any $v \in V$ satisfies the $a_\pw$ orthogonality $v - \mathcal{R}_h \I_h v \perp W_\mathrm{nc}$.
\end{lemma}

\begin{proof}
	For any $\varphi_{nc} \in W_\mathrm{nc}$, $\div_\pw \mathbb{C} \varepsilon_\pw(\varphi_{nc}) \in P_{k-1}(\Mcal)^n$ and $[\mathbb{C} \varepsilon_\pw(\varphi_{nc})]_S \nu_S \in P_k(S)^n$ along $S \in \Sigma$.
	By choosing $v_h \coloneqq \I_h v$ in \eqref{def:pot-rec-el-1}, we observe that $v_\Mcal = \Pi_\Mcal^{k+1} v$ and $v_S = \Pi_S^k v$ on the right-hand side of \eqref{def:pot-rec-el-1} can be replaced by $v$. A piecewise integration by parts of the resulting term shows $(\mathbb{C} \varepsilon_\pw(\mathcal{R}_h \I_h v),\varepsilon_\pw(\varphi_{nc}))_{L^2(\Omega)} = (\mathbb{C} \varepsilon(v), \varepsilon_\pw(\varphi_{nc}))_{L^2(\Omega)}$.
\end{proof}

The discrete problem \eqref{def:discrete-problem} is defined with
\begin{align*}
	a_h(u_h,v_h) \coloneqq a_\pw(\mathcal{R}_h u_h,\mathcal{R}_h v_h), \quad b_h(u_h, v_h) \coloneqq (u_\Mcal, v_\Mcal)_{L^2(\Omega)},
\end{align*}
and $\s_h(u_h, v_h) \coloneqq \sum_{K \in \mathcal{M}} \s_K(u_h,v_h)$ with
\begin{align*}
	&\s_K(u_h, v_h) \coloneqq 2\mu\sigma \int_K h_K^{-2} (u_K - \mathcal{R}_h u_h) \cdot (v_K - \mathcal{R}_h v_h) \d{x}\nonumber\\
	&\quad + 2\mu\sigma \sum_{S \in \Sigma(K)} \ell(S,K)^{-1} \int_S \Pi_S^k(u_S - \mathcal{R}_h u_h|_K) \cdot \Pi_S^k (v_S - \mathcal{R}_h v_h|_K) \d{s}
\end{align*}
for any $u_h = (u_\Mcal, u_\Sigma), v_h = (v_\Mcal, v_\Sigma) \in V_h$, where $\ell(S,K)$ is the weight from \eqref{def:weight}.

\begin{lemma}[coercivity]\label{lem:coercivity-le}
	The bilinear form $a_h + \s_h$ is a scalar product in $V_h$.
	In particular, there are $\mathrm{dim}(P_{k+1}(\Mcal)^n)$ finite discrete eigenvalues of \eqref{def:discrete-problem}. \qed
\end{lemma}

\begin{proof}
	Let $v_h = (v_\Mcal, v_\Sigma) \in V_h$ with $\|v_h\|_{a_h}^2 + \|v_h\|_{\s_h}^2 = 0$ be given. From $\|v_h\|_{a_h} = 0$, we deduce that $\mathcal{R}_h v_h$ is a piecewise rigid body motion and, in particular, piecewise affine.
	Since $k \geq 1$, $\|v_h\|_{\s_h} = 0$ implies that the traces $\mathcal{R}_h v_h$ coincide with $v_\Sigma$ and $v_\Mcal = \mathcal{R}_h v_h$. Therefore, $\mathcal{R}_h v_h$ is a continuous linear function with homogenous boundary data. This concludes $\mathcal{R}_h v_h = 0$ and so, $v_h = 0$.
\end{proof}

\begin{remark}[$\kappa$-robustness and spectral correctness]
	Utilizing the arguments of \cite{CarstensenTran2024}, we can prove the robustness of this method with respect to $\kappa \to \infty$ on regular triangulations $\Mcal$ into simplices in the following sense.
	Given $f \in L^2(\Omega)^n$, let $u \in V$ solve
	\begin{align*}
		-\div\,\sigma = f \text{ in } \Omega, \quad\sigma = \mathbb{C} \varepsilon(u), \quad \sigma \nu = 0 \text{ on } \Gamma_\mathrm{N}.
	\end{align*}
	Then the discrete solution $u_h \in V_h$ to $a_h(u_h,v_h) + \s_h(u_h,v_h) = (f,v_\Mcal)_{L^2(\Omega)}$ for any $v_h = (v_\Mcal,v_\Sigma) \in V_h$ satisfies
	\begin{align*}
		&\|\sigma - \sigma_h\|_{L^2(\Omega)} + \|\I_h u - u_h\|_{a_h} + \|\I_h u - u_h\|_{\s_h}\\
		&\qquad \lesssim \min_{p_{k+1} \in P_{k+1}(\Mcal)^n} \|\sigma - \mathbb{C} \varepsilon_\pw(p_{k+1})\|_{L^2(\Omega)} + \mathrm{osc}(f,\Mcal),
	\end{align*}
	where $\sigma_h \coloneqq \mathbb{C} \varepsilon_\pw(\mathcal{R}_h u_h)$ and $\mathrm{osc}(f,\Mcal) \coloneqq \|h_\Mcal(1 - \Pi_\Mcal^{k+1}) f\|_{L^2(\Omega)}$.
	The hidden constant in $\lesssim$ is independent of $\kappa$ and the mesh-size. The spectral correctness of the HHO eigensolver of this section then follows from \cite{CaloCicuttinDengErn2019}.
\end{remark}

\subsection{Lower eigenvalue bounds}\label{sec:lower-eig-bounds-LE}
This subsection verifies \Cref{A}--\Cref{C} and derives GLB for the linear elasticity eigenvalue problem.
The Galerkin projection $\mathrm{G}_h v \in W_\mathrm{nc}$ of $v \in V_\mathrm{nc}$ is the unique solution to \eqref{def:Galerkin} with, for any $K \in \Mcal$,
\begin{align}\label{def:Galerkin-LE-fixed-dof}
	\int_K \mathrm{G}_h v \d{x} = \int_K v \d{x} \quad\text{and}\quad \int_K \D_\mathrm{ss} \mathrm{G}_h v_h \d{x} = \int_K \D_\mathrm{ss} v \d{s}.
\end{align}
\begin{lemma}[$\mathcal{R}_h \circ \I_h = \mathrm{G}_h$]\label{lem:RI=G-linear-elasticity}
	The identity \eqref{eq:R-I=G} holds in $V$.
\end{lemma}
\begin{proof}
	Given $v \in V$, \Cref{lem:projection-linear-elasticity} implies that $\mathcal{R}_h \I_h v$ satisfies \eqref{def:Galerkin}. 
	It remains to verify \eqref{def:Galerkin-LE-fixed-dof}.
	For any $K \in \Mcal$, $\int_K \mathcal{R}_h \I_h v \d{x} = \int_K \Pi_\Mcal^{k+1} v \d{x} = \int_K v \d{x}$ and $\int_K \D_\mathrm{ss} \mathcal{R}_h \I_h v \d{x} = \sum_{S \in \Sigma(K)} \int_S \mathrm{asym}(\nu_K \otimes v) \d{s} = \int_K \D_\mathrm{ss} v \d{x}$ from \eqref{def:pot-rec-el-2} and an integration by parts. Hence, $\mathcal{R}_h \I_h v = \mathrm{G}_h v$.
\end{proof}
Given $K \in \Mcal$, let $c_\mathrm{Korn}(K)$ denote the Korn constant in the inequality
\begin{align}\label{ineq:Korn}
	\|\D v\|_{L^2(K)} \leq c_\mathrm{Korn}(K)\|\varepsilon(v)\|_{L^2(K)}
\end{align}
for any $v \in H^1(K)^n$ with $\int_K \D_\mathrm{ss} v \d{x} = 0$ \cite{Brenner2004}. The computation of GLB for linear elasticity requires the following result in the local space $V(K) \coloneqq \{v \in H^1(K)^n : \int_K v \d{x} = 0 \text{ and } \int_K \D_\mathrm{ss} v \d{x} = 0\}$.
\begin{lemma}[local bound in linear elasticity]\label{lem:comp-emb-le}
	Given a convex polyhedra $K \subset \R^n$, there exists a constant $\gamma(K) \geq c_\mathrm{Korn}^{-2}(K)(\pi^{-2} + c_\mathrm{st})^{-1}$ with, for any $v \in V(K)$,
	\begin{align*}
		\sum_{S \in \Sigma(K)} \ell(S,K)^{-1}\|v\|_{L^2(S)}^2 + h_K^{-2}\|v\|_{L^2(K)}^2 \leq \gamma(K)^{-1} \|\varepsilon(v)\|^2_{L^2(K)}.
	\end{align*}
	The constant $\gamma(K)$ may depend on the shape but not on the diameter of $K$.
\end{lemma}
\begin{proof}
	The assertion follows from \Cref{lem:trace-inequality}, a Poincar\'e, and Korn inequality.
\end{proof}
\Cref{lem:RI=G-linear-elasticity}--\ref{lem:comp-emb-le} allow for the verification of \ref{A}--\ref{C} with the constants $c_\mathrm{Korn} \coloneqq \max_{K \in \Mcal} c_\mathrm{Korn}(K)$ and $\gamma \coloneqq \min_{K \in \mathcal{M}} \gamma(K)$. 
\begin{corollary}[GLB for linear elasticity]\label{cor:linear-elasticity}
	It holds
	\begin{align*}
		\min\{1,1/(\gamma^{-1}\sigma + c_\mathrm{Korn}^2 h_\mathrm{max}^2\lambda_h(j)/(2\pi^2\mu))\}\lambda_h(j) \leq \lambda(j)
	\end{align*}
\end{corollary}
\begin{proof}
	Given $v \in V$,
	the Pythagoras theorem with the $a_\pw$ orthogonality $v - \mathrm{G}_h v \perp W_\mathrm{nc}$ from \Cref{lem:projection-linear-elasticity} and \Cref{lem:RI=G-linear-elasticity} proves \Cref{A}, namely,
	\begin{align*}
		\|\mathrm{G}_h v\|_{a_\pw}^2 = \|v\|_a^2 - \|v - \mathrm{G}_h v\|_{a_\pw}^2.
	\end{align*}
	Since $v - \mathrm{G}_h v \in V(K)$ for any $K \in \Mcal$, \Cref{lem:RI=G-linear-elasticity}--\ref{lem:comp-emb-le}
	and the best approximation property of $L^2$ projections provide
	\begin{align}\label{ineq:stabilization-le}
		s_K(\I_h v, \I_h v) &\leq 2 \mu \sigma \sum_{S \in \Sigma(K)} \ell(S,K)^{-1}\|v - \mathrm{G}_h v\|^2_{L^2(S)} + 2 \mu \sigma h_K^{-2}\|v - \mathrm{G}_h v\|_{L^2(K)}^2\nonumber\\
		&\leq 2\gamma(K)^{-1}\mu\sigma\|\varepsilon(v - \mathrm{G}_h v)\|^2_{L^2(K)}. 
	\end{align}
	The sum of this over all $K \in \Mcal$ yields 
	\Cref{B} with the constant $\alpha \coloneqq \gamma^{-1}\sigma$.
	As for the Laplace equation, \Cref{C} follows from the Pythagoras theorem and the Korn inequality \eqref{ineq:Korn} with the constant $\beta = c_{\mathrm{Korn}}^2 h_\mathrm{max}^2/(2\pi^2\mu)$. The proof of \Cref{thm:dLEB} applies verbatim and concludes the proof.
\end{proof}

\begin{remark}[alternative eigensolver]\label{rem:alt-eig-le}
	The following alternative approach utilizes an additional reconstruction for the divergence as suggested in \cite{DiPietroErn2015}.
	Given $v_h = (v_\Mcal, v_\Sigma) \in V_h$, the divergence reconstruction $\div_h v_h \in P_k(\Mcal)$ uniquely solves
	\begin{align*}
		(\div_h v_h, p_k)_{L^2(\Omega)} = -(v_\Mcal, \nabla_\pw p_k)_{L^2(\Omega)} + \sum_{S \in \Sigma} (v_S \cdot \nu_S, [p_k]_S)_{L^2(S)}
	\end{align*}
	for any $p_k \in P_k(\Mcal)$. The discrete problem \eqref{def:discrete-problem} is defined as in \Cref{sec:le-eigensolver} with the following modifications:
	$a_\pw(u_\mathrm{nc}, v_\mathrm{nc}) \coloneqq (\varepsilon_\pw(u_\mathrm{nc}), \varepsilon_\pw(v_\mathrm{nc}))_{L^2(\Omega)}$ and
	$a_h$ is replaced by
	$$a_h(u_h,v_h) \coloneqq 2\mu a_\pw(\mathcal{R}_h u_h,\mathcal{R}_h v_h)_{L^2(\Omega)} + \kappa(\div_h u_h, \div_h v_h)_{L^2(\Omega)},$$
	where $\mathcal{R}_h$ is the potential reconstruction from \eqref{def:pot-rec-ex-2-1} below, similar to \cite[Section 3.2]{DiPietroErn2015}.
	For any $v \in V$, the $a_\pw$ orthogonality $v - \mathcal{R}_h \I_h v \perp W_\mathrm{nc}$ \cite[Eq.~(19)]{DiPietroErn2015} implies \eqref{eq:R-I=G}.
	This and
	the Pythagoras theorem with the $L^2$ orthogonality $\div\,v - \div_h \I_h v \perp P_k(\Mcal)$ \cite[Proposition 3]{DiPietroErn2015} show \ref{A} with
	\begin{align*}
		\|\I_h v\|_{a_h}^2 &= 2\mu\|v\|^2_{a_\pw} - 2\mu\|v - \mathrm{G}_h v\|_{a_\pw} + \kappa\|\div v\|^2_{L^2(\Omega)} - \kappa\|(1-\Pi_\Mcal^k)\div v\|_{L^2(\Omega)}^2\\
		&\leq 2\mu\|v\|^2_a - 2\mu\|v-\mathrm{G}_h v\|_{a_\pw}^2.
	\end{align*}
	The condition \ref{B} follows from \eqref{ineq:stabilization-le}; \ref{C} is a consequence of the Pythagoras theorem and the Korn inequality \eqref{ineq:Korn}. It turns out that \ref{A}--\ref{C} are valid with the same constants $\alpha$ and $\beta$ from \Cref{cor:linear-elasticity}, whence \Cref{cor:linear-elasticity} holds verbatim.
\end{remark}

In practice, the computation of the GLB in \Cref{cor:linear-elasticity} requires an upper bound of the Korn constant in convex polyhedra. This is closely related to the continuity constant of a right-inverse of the divergence operator and is available in 2d from \cite{CostabelDauge2015,Gallistl2023}. For any $K \in \Mcal$, let $z_K$ be the midpoint of a largest ball inscribed in $K$ with nodes $x_1, \dots, x_m$. Define the geometric parameter
\begin{align*}
	\varrho(K) \coloneqq \frac{\mathrm{dist}(z_K, \partial K)}{\max_{j = 1,\dots, m} |z_K - x_j|} \quad\text{and}\quad \varrho \coloneqq \max_{K \in \mathcal{M}} \varrho(K),
\end{align*}
then \cite[Lemma 6.2]{Gallistl2023} provides the upper bound
\begin{align*}
	c_\mathrm{Korn} \leq \sqrt{1 + \frac{4}{\varrho^2}(1 + \sqrt{1 - \varrho^2})} \eqqcolon \hat{c}_\mathrm{Korn}
\end{align*}
of the Korn constant.
For right-isosceles triangles, $\hat{c}_\mathrm{Korn} = 7.318$ \cite{Gallistl2023}.

\begin{figure}[ht]
	\begin{minipage}[b]{0.49\linewidth}
		\centering
		\includegraphics[height=5cm]{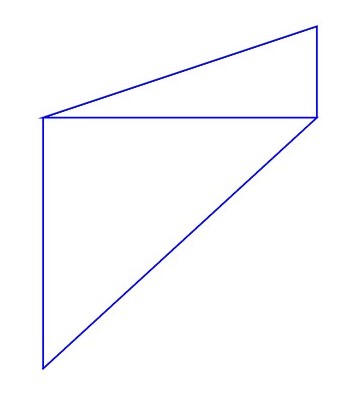}
	\end{minipage}\hfill
	\begin{minipage}[b]{0.49\linewidth}
		\centering
		\includegraphics[height=5cm]{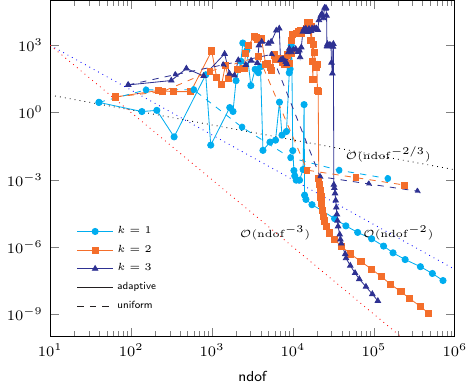}
	\end{minipage}
	\captionsetup{width=1\linewidth}
	\caption{(a) Initial triangulation and (b) convergence history plot of $\lambda_\mathrm{C}(1) - \mathrm{GLB}(1)$ for the linear elasticity eigenvalue problem in \Cref{sec:ex-linear-elast}}
	\label{fig:linear-elasticity}
\end{figure}

\subsection{Computer benchmark}\label{sec:ex-linear-elast}
We approximate the first linear elasticity eigenvalue in the Cook's membrane $\Omega \coloneqq \mathrm{conv}\{(0,0), (48,44), (48,60)$, $(0,44)\}$ with the Dirichlet boundary $\Gamma_\mathrm{D} \coloneq \mathrm{conv}\{(0,0), (0,44)\}$, Neumann boundary $\Gamma_\mathrm{N} \coloneqq \partial \Omega \setminus \Gamma_\mathrm{D}$, and parameters $\mu = 1/2$ and $\kappa = 1000$. The initial triangulation of $\Omega$ is displayed in \Cref{fig:linear-elasticity}(a). The reference value $\lambda(1) = 2.9020 \times 10^{-4}$ stems from the bound $\mathrm{GLB}(1) \leq \lambda(1) \leq \lambda_h(1)$.
We set $\sigma \coloneqq 2^{-1}\hat{c}_\mathrm{Korn}^{-2}(\pi^{-2} + c_\mathrm{tr})^{-1} = 0.016049 \leq \gamma/2$ so that $\min\{1,1/(1/2 + \hat{c}_\mathrm{Korn}^2 h_\mathrm{max}^2\lambda_h(j)/\pi^2)\}\lambda_h(j) \leq \lambda(j)$ from \Cref{cor:linear-elasticity}. The computation of $\lambda_\mathrm{C}(1)$ and $u_\mathrm{C}(1)$ follows \Cref{sec:num_ex_laplace}. For the first discrete eigenvalue $u_h(1)$, the bound
\begin{align*}
	\lambda_\mathrm{C}(1) - \lambda_h(1) \leq \|u_\mathrm{C}(1) - \mathcal{R}_h u_h(1)\|^2_{a_\pw} + \|u_h(1)\|_{\s_h}^2
\end{align*}
similar to \eqref{def:eta} motivates a refinement indicator by localization of the right-hand side.
\Cref{fig:linear-elasticity}(b) displays the convergence history of $\lambda_\mathrm{C}(1) - \mathrm{GLB}(1)$ and the observations of 
\Cref{sec:num_ex_laplace} apply.

\section{Local embedding in linear elasticity}\label{sec:compact-emb}
The quality of lower eigenvalue bounds depends on the accuracy of constants in \ref{B}--\ref{C}. For the HHO eigensolvers of this paper however, \ref{B} also directly influences the
numerical scheme in the choice of $\sigma$ for the stabilization so that $\alpha < 1$.
An overestimation in trace or Korn inequalities leads to inaccurate lower eigenvalue bounds on coarse meshes in all computer benchmarks of the previous three sections.
This section proposes eigensolvers to compute improved guaranteed bounds of the relevant constants in \Cref{thm:dLEB}, which can significantly reduce the size of the preasymptotic range.
We focus on constants arising in the linear elasticity eigenvalue problem of \Cref{sec:linear-elasticity}.
Given a convex polyhedra $\Omega$ of diameter $d \coloneqq \mathrm{diam}(\Omega)$ with the set of sides $\mathcal{E}$,
define the piecewise constant weight function $\ell(\partial \Omega) \in P_0(\mathcal{E})$ by
\begin{align*}
	\ell(\partial\Omega)|_S = \ell(S,\Omega) \quad\text{for any } S \in \mathcal{E}
\end{align*}
with $\ell(S,\Omega)$ from \eqref{def:weight}.
We seek lower bounds $\gamma_*(\Omega)$ of
\begin{align}\label{def:RQ-comp-emb}
	\gamma(\Omega) = \min_{v \in V\setminus \{0\}} \|v\|_a^2/\|v\|^2_b
\end{align}
in the space $V = \{v \in H^1(\Omega)^n : \int_\Omega v \d{x} = 0 \text{ and } \int_\Omega \D_\mathrm{ss} v \d{x} = 0\}$ with
\begin{align*}
	a(v,v) \coloneqq (\varepsilon(v),\varepsilon(v))_{L^2(\Omega)} \quad\text{and}\quad
	b(v,v) \coloneqq (\ell(\partial \Omega)^{-1} v,v)_{L^2(\partial \Omega)} + d^{-2}(v,v)_{L^2(\Omega)}.
\end{align*}
Notice that the minimum in \eqref{def:RQ-comp-emb} is attained because every bounded sequence in the set $S(V) \coloneqq \{v \in V : \|v\|_b = 1\}$ has a strong limit in $S(V)$ by compact embeddings.
Furthermore, $\gamma(\Omega)^{-1}$ is the best possible constant in $\|v\|_b^2 \leq \gamma(\Omega)^{-1}\|v\|_a^2$ for all $v \in V$ and is utilized in \Cref{cor:linear-elasticity} for GLB of the linear elasticity eigenvalue problem.

\subsection{HHO eigensolver}\label{sec:HHO-comp-emb}
Set $V_\mathrm{nc} \coloneqq H^1(\Mcal)^n$, $a_\pw \coloneqq (\varepsilon_\pw (\bullet), \varepsilon_\pw (\bullet))_{L^2(\Omega)}$, and $W_\mathrm{nc} \coloneqq P_{k+1}(\Mcal)^n$.
For any $v_h = (v_\Mcal, v_\Sigma) \in P_{k+1}(\Mcal)^n \times P_{k+1}(\Sigma)^n$, the potential reconstruction $\mathcal{R}_h v_h \in W_\mathrm{nc}$ of $v_h$ from \cite{DiPietroErn2015} satisfies, for any $\varphi_{k+1} \in W_\mathrm{nc}$, that
\begin{align}\label{def:pot-rec-ex-2-1}
	&a_\pw(\mathcal{R}_h v_h, \varphi_{k+1}) \nonumber\\
	&\quad= - (v_\Mcal, \div_\pw \varepsilon_\pw (\varphi_{k+1}))_{L^2(\Omega)} + \sum_{S \in \Sigma} (v_S, [\varepsilon_\pw (\varphi_{k+1})]_S \nu_S)_{L^2(S)}.
\end{align}
This defines $\mathcal{R}_h v_h$ uniquely up to rigid body motions. The associated degrees of freedom are fixed by \eqref{def:pot-rec-el-2}.
Given $k \geq 0$,
let $V_h \coloneqq \{v_h = (v_\Mcal, v_\Sigma) \in P_{k+1}(\Mcal)^n \times P_{k+1}(\Sigma)^n : \int_\Omega v_\Mcal = 0 \text{ and } \int_\Omega \D_\mathrm{ss} \mathcal{R}_h v_h \d{x} = 0\}$ denote the discrete ansatz space.
We claim that $\I_h v \coloneqq (\Pi_\Mcal^{k+1} v, \Pi_\Sigma^{k+1} v) \in V_h$ for any $v \in V$.
In fact, 
the fixed degrees of freedom in \eqref{def:pot-rec-el-2} imply
$\int_K \D_\mathrm{ss} \mathcal{R}_h \I_h v \d{x} = \int_K \D_\mathrm{ss} v \d{x}$ for any $K \in \Mcal$ in the proof of \Cref{lem:RI=G-linear-elasticity}, whence
\begin{align*}
	\int_\Omega \D_\mathrm{ss} \mathcal{R}_h \I_h v \d{x} = \int_\Omega \D_\mathrm{ss} v \d{x} = 0.
\end{align*}
Since $\int_\Omega \Pi_\Mcal^{k+1} v \d{x} = 0$ from $\int_\Omega v \d{x} = 0$, this shows
$\I_h v \in V_h$ and so, $\I_h$ is an interpolation from $V$ onto $V_h$.
The discrete problem \eqref{def:discrete-problem} is defined with
\begin{align*}
	a_h(u_h,v_h) &\coloneqq a_\pw(\mathcal{R}_h u_h, \mathcal{R}_h v_h),\\
	b_h(u_h,v_h) &\coloneqq (\ell(\partial \Omega)^{-1} u_\Sigma, v_\Sigma)_{L^2(\partial \Omega)} + d^{-2}(u_\Mcal, v_\Mcal)_{L^2(\Omega)},
\end{align*}
and the stabilization
$\s_h(u_h, v_h) \coloneqq \sum_{K \in \mathcal{M}} \s_K(u_h,v_h)$,
\begin{align}\label{def:comp-emb-s}
	\s_K(u_h, v_h) &\coloneqq \sigma \int_K h_K^{-2} (u_K - \mathcal{R}_h u_h) \cdot (v_K - \mathcal{R}_h v_h) \d{x}\nonumber\\
	&\quad + \sigma\sum_{S \in \Sigma(K)} \ell(S,K)^{-1} \int_S (u_S - \mathcal{R}_h u_h|_K) \cdot (v_S - \mathcal{R}_h v_h|_K) \d{s}
\end{align}
for any $u_h = (u_\Mcal, u_\Sigma), v_h = (v_\Mcal, v_\Sigma) \in V_h$.
It is straightforward to verify the following result as in the proof of \Cref{lem:coercivity-le}.
\begin{lemma}[coercivity]\label{lem:well-posedness-compact-embedding}
	The bilinear form $a_h + \s_h$ is a scalar product in $V_h$. In particular, there are $\mathrm{dim}(P_{k+1}(\Mcal)^{n}) + \mathrm{dim}(P_{k+1}(\Sigma(\partial \Omega))^{n}) - M$ finite discrete eigenvalues of \eqref{def:discrete-problem}, where $M$ is the number of rigid-body motions in $n$-d. \qed
\end{lemma}
\begin{remark}[convergence for principle eigenvalue]
	We briefly outline a proof for the convergence of the first discrete eigenvalue $\lambda_h(1)$ towards $\gamma(\Omega)$ from \eqref{def:RQ-comp-emb} as $h \to 0$ on uniformly refined meshes. Let $u_h(1) = u_h = (u_\Mcal,u_\Sigma) \in V_h$ denote the first eigenvector with the normalization $\|u_h\|_{b_h} = 1$.
	Discrete functional analytic tools in the HHO methodology
	\cite[Theorem 6.8]{DiPietroDroniou2020} show the existence of a limiting function $v \in V$ with
	\begin{align*}
		&\varepsilon_\pw(\mathcal{R}_h u_h) \rightharpoonup \varepsilon(v) \text{ in } L^2(\Omega)^{n \times n},\\ 
		&u_\Mcal \to v \text{ in } L^2(\Omega)^n,
		\quad\text{and}\quad u_\Sigma|_{\partial \Omega} \to v \text{ on } L^2(\partial \Omega)^n
	\end{align*}
	up to a not relabelled subsequence,
	whence $\|v\|_b = 1$. As a convex functional, $\|\bullet\|_{a_h}^2$ is weakly lower semi-continuous and so,
	\begin{align*}
		\gamma(\Omega) \leq \|v\|_a^2 \leq \liminf_{h \to 0} \|u_h\|_{a_h}^2 \leq \liminf_{h \to 0} \lambda_h(1).
	\end{align*}
	On the other hand, let $u \in V$ denote the minimizer of \eqref{def:RQ-comp-emb}. For sufficiently small mesh-sizes $h$, $\|\mathrm{I}_h u\|_{b_h} \neq 0$ and the discrete function
	$v_h \coloneqq \mathrm{I}_h u/\|\mathrm{I}_h u\|_{b_h} \in V_h$ is well defined. 
	It is straightforward to verify that
	\begin{align*}
		\limsup_{h \to 0} \lambda_h(1) \leq \limsup_{h \to 0} \big(\|v_h\|_{a_h}^2 + \|v_h\|_{s_h}^2\big) = \|u\|_{a}^2 = \gamma(\Omega).
	\end{align*}
	The combination of the two previously displayed formula conclude $\lim_{h \to 0} \lambda_h(1) = \gamma(\Omega)$.
\end{remark}

\subsection{Lower eigenvalue bounds}
The conditions \ref{A}--\ref{C} are verified in this subsection, which lead to GLB for the first eigenvalue in \Cref{cor:comp-emb} below.
Unfortunately, it is not clear whether \eqref{def:RQ-comp-emb} can be written as a compact eigenvalue problem. Nevertheless, the proof of \Cref{thm:dLEB} carries over for the principle eigenvalue $\gamma(\Omega)$ of \eqref{def:RQ-comp-emb}.
The Galerkin projection $\mathrm{G}_h v \in W_\mathrm{nc}$ of $v \in V_\mathrm{nc}$ is defined by \eqref{def:Galerkin} with \eqref{def:Galerkin-LE-fixed-dof}.
Recall $\gamma = \min_{K \in \Mcal} \gamma(K)$ with $\gamma(K)$ from \Cref{lem:comp-emb-le}.
\begin{corollary}[GLB in local embeddings related to linear elasticity]\label{cor:comp-emb}
	It holds
	\begin{align*}
		\mathrm{GLB}(1) \coloneqq \min\{1,\gamma/\sigma\} \lambda_h(1) \leq \gamma(\Omega).
	\end{align*}
\end{corollary}
\begin{proof}
	The verification of \eqref{eq:R-I=G} is similar to \Cref{lem:RI=G-linear-elasticity}, cf.~also \cite{DiPietroErn2015}. This provides \ref{A} even with equality (instead of $\leq$).
	Given $v \in V$,
	the identity
	\begin{align*}
		\s_K(\I_h v, \I_h v) = \sigma \sum_{S \in \Sigma(K)} \ell(S,K)^{-1}\|v - \mathrm{G}_h v\|^2_{L^2(S)} + \sigma h_K^{-2}\|v - \mathrm{G}_h v\|_{L^2(K)}^2
	\end{align*}
	for any $K \in \Mcal$ from \Cref{lem:RI=G-linear-elasticity} and \Cref{lem:comp-emb-le} imply \ref{B} with $\alpha = \gamma^{-1}\sigma$. 
	Since $\ell(\partial \Omega)$ is piecewise constant,
	the Pythagoras theorem with the orthogonality of $L^2$ projections establish
	\begin{align*}
		\|\I_h v\|_{b_h}^2 &= \sum_{S \in \Sigma(\partial \Omega)} \ell(\partial \Omega)|_S^{-1} \|\Pi_\Sigma^{k+1} v\|^2_{L^2(S)} - d^{-2}\|\Pi_\Mcal^{k+1} v\|_{L^2(\Omega)}^2\\
		&= \|v\|_b^2 - \sum_{S \in \Sigma(\partial \Omega)} \ell(\partial \Omega)|_S^{-1} \|(1 - \Pi_\Sigma^{k+1}) v\|^2_{L^2(S)} - d^{-2}\|(1 - \Pi_\Mcal^{k+1}) v\|_{L^2(\Omega)}^2\\
		&\geq \|v\|_b^2 - \sum_{S \in \Sigma(\partial \Omega)} \ell(\partial \Omega)|_S^{-1} \|v - \mathrm{G}_h v\|^2_{L^2(S)} - d^{-2}\|v - \mathrm{G}_h v\|_{L^2(\Omega)}^2
	\end{align*}
	This, \eqref{ineq:trace-steklov}, the Poincar\'e, and Korn inequality show \ref{C} with $\beta = c_\mathrm{Korn}^2(h_\mathrm{max}^2/\pi^2 + \beta_\mathrm{st})$ and $\beta_\mathrm{st}$ from \eqref{def:beta-st}.
	The arguments in the proof of \Cref{thm:dLEB} yields
	\begin{align*}
		\min\{1,1/(\gamma^{-1} \sigma + c_\mathrm{Korn}^2 (h_\mathrm{max}^2/\pi^2 + \beta_\mathrm{st})\lambda_h(1))\}\lambda_h(1) \leq \gamma(\Omega).
	\end{align*}
	Given a positive parameter $t > 0$, consider the domain $\Omega_t \coloneqq t \Omega$ with the partition $\Mcal_t \coloneqq \{t K: K \in \Mcal\}$.
	Since the continuous and discrete Rayleigh quotients are invariant under scaling, the continuous and discrete eigenvalues in $\Omega_t$ coincide with those in $\Omega$.
	On the mesh $\Mcal_t$, the arguments from above lead to the GLB
	\begin{align*}
		\min\{1,1/(\gamma^{-1} \sigma + c_\mathrm{Korn}^2 (t^2h_\mathrm{max}^2/\pi^2 + t\beta_\mathrm{st})\lambda_h(1))\}\lambda_h(1) \leq \gamma(\Omega)
	\end{align*}
	due to the invariance of $c_\mathrm{Korn}$ and the scaling $\beta_\mathrm{st} \approx h_\mathrm{max}$.
	The limit of this as $t \to 0$ concludes the proof.
\end{proof}

\begin{remark}[WG alternative]
	The polynomial degrees in the ansatz space $V_h$ of the eigensolver in \Cref{sec:HHO-comp-emb} are also utilized in earlier WG methods, e.g.~\cite{MuWangYe2015}, with the stabilization
		\begin{align*}
			\widetilde{s}_K(u_h,v_h) \coloneqq \sum_{S \in \Sigma(K)} h_S^{-1} \int_S (u_S - u_K) \cdot (v_S - v_K) \d{s} 
		\end{align*}
		for any $u_h = (u_\Mcal, u_\Sigma), v_h = (v_\Mcal, v_\Sigma) \in V_h$ instead of \eqref{def:comp-emb-s}. Note that this stabilization also allows for lower eigenvalue bounds. In fact, \cite[Theorem 3.1]{CarstensenZhaiZhang2020} shows the existence of a constant $C_\mathrm{st}$ such that
		\begin{align*}
			\|\mathrm{I}_h v\|_{\widetilde{s}_h}^2 \leq C_\mathrm{st}\|v - \mathrm{G}_h v\|_{a_\pw}^2 \quad\text{for any } v \in V,
		\end{align*}
		replacing \ref{B}. However, guaranteed upper bounds for $C_\mathrm{st}$ are not known for $k \geq 1$.
\end{remark}

\begin{table}
	\centering
	\begin{tabular}{c|c|c}
		Iteration & $\sigma$ & $\gamma_*$\\
		\hline
		1 & 0.0358473 & 0.286114\\
		2 & 0.286114 & 2.2488\\
		3 & 2.2488 & 6.45254\\
		4 & 6.45254 & 7.01185\\
		5 & 7.01185 & 7.02938
	\end{tabular}
	\captionsetup{width=1\linewidth}
	\caption{Lower bounds of $\gamma$ for the reference triangle in \Cref{sec:ex-comp-emb}}
	\label{tab:leb-comp-emb}
\end{table}

\subsection{Computer benchmark}\label{sec:ex-comp-emb}
Lower bounds $\gamma_*(\Omega) = \mathrm{GLB}(1)$ of $\gamma(\Omega)$ in the reference triangle $\Omega = \mathrm{conv}\{(0,0),(1,0),(0,1)\}$ are computed on a fixed mesh $\Mcal$ created by three successive uniform refinements of the reference triangle, where each refinement divides every cell of $\Mcal$ into four congruent cells by connecting the mid points on the three sides.
Notice that the constant $\gamma(\Omega)$ is invariant under translation, scaling, and half rotation.
Since every triangle in $\Mcal$ is the image of $\Omega$ under a combination of these transformations,
$\gamma(K) = \gamma(\Omega)$ for all $K \in \Mcal$ in \Cref{lem:comp-emb-le}. This leads to
\begin{align}\label{ineq:comp-emb-tri}
	\gamma_*(\Omega) = \min\{1,\gamma(\Omega)/\sigma\} \lambda_h(1) \leq \gamma(\Omega)
\end{align}
in \Cref{cor:comp-emb}.
We set $\sigma \coloneqq \hat{c}_\mathrm{Korn}^{-2}(\pi^{-2} + c_\mathrm{tr})^{-1} = 0.0358473 \leq \gamma(\Omega)$ so that $\lambda_h(j) \leq \lambda(j)$ in \eqref{ineq:comp-emb-tri}.
From the numerical experiments in the previous section, only rough bounds can be expected.
For the choice $k = 1$, we obtain the lower bound $\gamma_*(\Omega) = 0.286114$. (For reference, $\gamma(\Omega) \geq 7$ below.)
Since the computed GLB $\gamma_*(\Omega)$ is larger than the initial bound
$\hat{c}_\mathrm{Korn}^{-2}(\pi^{-2} + c_\mathrm{tr})^{-1} = 0.0358473$, we repeat the computation with $\sigma = \gamma_*(\Omega) = 0.286114$. This process can be successively repeated by updating $\sigma = \gamma_*(\Omega)$ after each iteration.
\Cref{tab:leb-comp-emb} displays the computed GLB and provides $7 \leq \gamma(\Omega)$ in the reference triangle.

\begin{figure}
	\begin{minipage}{0.33\linewidth}
		\centering
		\includegraphics[height=2cm]{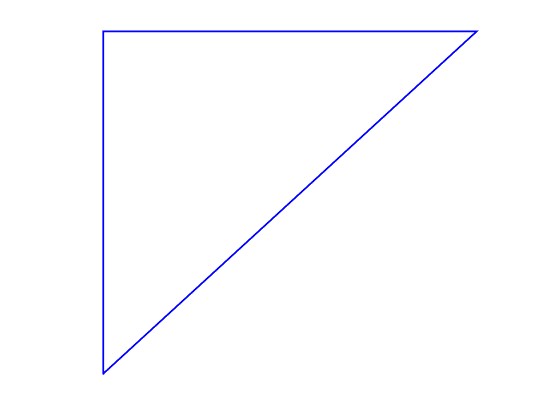}
		\captionsetup{width=1\linewidth}
		\caption*{$6.9898 \leq \gamma$}
	\end{minipage}\hfill
	\begin{minipage}{0.33\linewidth}
		\centering
		\includegraphics[height=2cm]{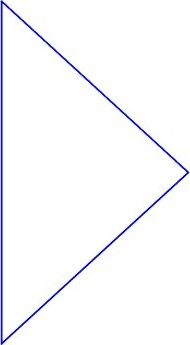}
		\captionsetup{width=1\linewidth}
		\caption*{$6.9344 \leq \gamma$}
	\end{minipage}
	\begin{minipage}{0.33\linewidth}
		\centering
		\includegraphics[height=2cm]{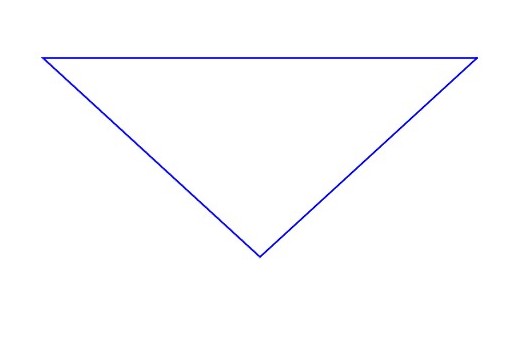}
		\captionsetup{width=1\linewidth}
		\caption*{$7.11834 \leq \gamma$}
	\end{minipage}
\end{figure}
\begin{figure}
	\begin{minipage}{0.33\linewidth}
		\centering
		\includegraphics[height=2cm]{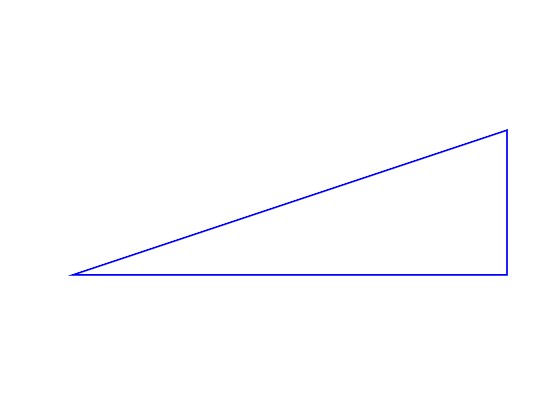}
		\captionsetup{width=1\linewidth}
		\caption*{$5.52756 \leq \gamma$}
	\end{minipage}\hfill
	\begin{minipage}{0.33\linewidth}
		\centering
		\includegraphics[height=2cm]{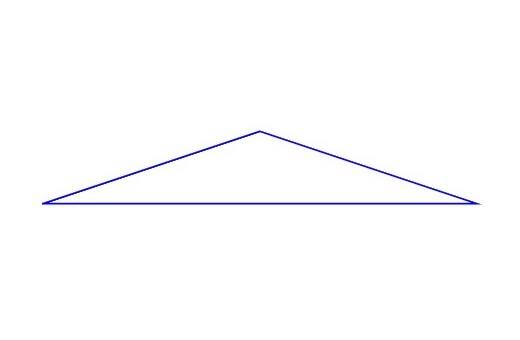}
		\captionsetup{width=1\linewidth}
		\caption*{$7.7239 \leq \gamma$}
	\end{minipage}
	\begin{minipage}{0.33\linewidth}
		\centering
		\includegraphics[height=2cm]{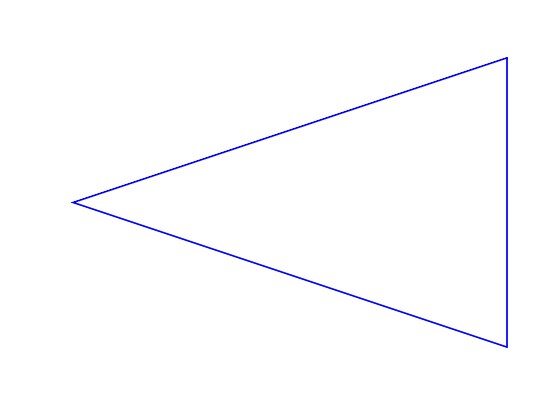}
		\captionsetup{width=1\linewidth}
		\caption*{$5.56383 \leq \gamma$}
	\end{minipage}
	\captionsetup{width=1\linewidth}
	\caption{Lower bounds of $\gamma$ for different triangular shapes in \Cref{sec:ex-le-impr}}
	\label{fig:comp-emb-diff}
\end{figure}
\begin{figure}
	\centering
	\includegraphics[height=8cm]{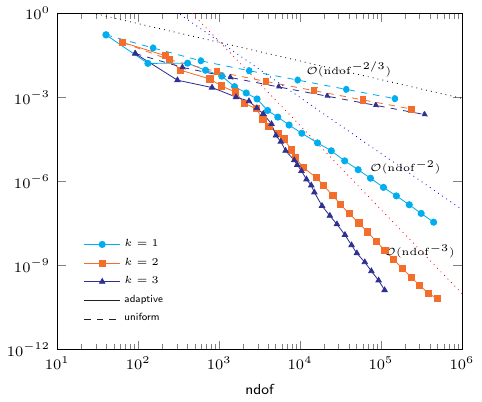}
	\captionsetup{width=1\linewidth}
	\caption{Improved convergence history plot of $\lambda_\mathrm{C}(1) - \mathrm{GLB}(1)$ for the linear elasticity eigenvalue problem in \Cref{sec:ex-le-impr}}
	\label{fig:linear-elasticity-impr}
\end{figure}

\subsection{A revisit to linear elasticity}\label{sec:ex-le-impr}
In the setting of \Cref{sec:ex-linear-elast}, we compute GLB of $\gamma(K)$
for all possible triangles $K$ that can occur from the newest vertex bisection as described in \Cref{sec:ex-comp-emb}. For the initial triangulation of the Cook's membrane displayed in \Cref{fig:linear-elasticity}(a), GLB need to be computed in six different triangles due to invariance of $\gamma(K)$ under scaling, translation, and half rotation. These triangles with a computed lower bound of $\gamma(K)$ are displayed in \Cref{fig:comp-emb-diff}. We obtain the improved bound $5.52 \leq \gamma$. (For comparison, $\hat{c}_\mathrm{Korn}^{-2}(\pi^{-2} + c_\mathrm{tr})^{-1} = 0.032098$.)
Since $\|v\|_{L^2(K)}^2 \leq \gamma(K)^{-1} h_K^2\|\varepsilon(v)\|_{L^2(K)}^2$ for any $v \in V(K)$, $\beta = \gamma^{-1} h_\mathrm{max}^2$ in \ref{C}.
The choice $\sigma = 2.76$ in \Cref{cor:linear-elasticity} leads to the GLB
\begin{align*}
	\mathrm{GLB}(j) = \min\{1,1/(1/2 + h_\mathrm{max}^2\lambda_h(j)/5.52)\}\lambda_h(j) \leq \lambda(j)
\end{align*}
for the benchmark in \Cref{sec:ex-linear-elast}. \Cref{fig:linear-elasticity-impr} displays the convergence history plot of $\lambda_\mathrm{C}(1) - \mathrm{GLB}(1)$ with improved accuracy in comparison to the results displayed in \Cref{fig:linear-elasticity}.

\subsection{Conclusions}
In all examples, adaptive computation recovers optimal convergence rates of the lower eigenvalue bounds towards the exact eigenvalue.
Accurate upper bounds of constants related to local embeddings can be computed and improve the quality of the lower eigenvalue bounds.
Additional applications beyond those presented in this paper include general scalar elliptic operators (instead of Laplace) and the biharmonic eigenvalue problem \cite{DongErn2022,LiangTran2025}.
Another application is the Stokes eigenvalue problem in primal form, where the vanishing divergence is enforced by using the reconstruction $\div_h$ from \Cref{rem:alt-eig-le} on the discrete level.
Since the conditions \ref{B}--\ref{C} arise from local estimates in practice, the analysis extends to piecewise constant $\sigma$ for improved handling of variable material parameters, e.g., in linear elasticity.

\printbibliography

\end{document}